\documentclass{amsart}
\usepackage{fullpage}
\usepackage{array}
\usepackage{mathrsfs}
\usepackage{hyperref}
\usepackage{graphicx,color}
\usepackage{float}

\newcommand{\varied}

\usepackage[pagewise]{lineno}

\newcommand{\Idd}{\ensuremath{\textnormal{Id}}^{*}}

\newcommand{\var}{\ensuremath{\textnormal{var}}}
\newcommand{\varGstar}{\ensuremath{\textnormal{var}^{*}}}

\usepackage{longtable}

\newcommand{\black}{\color{black}}

\newtheorem{theorem}{Theorem}[section]
\newtheorem{example}[theorem]{Example}

\newtheorem{lemma}[theorem]{Lemma}
\newtheorem{corollary}[theorem]{Corollary}
\newtheorem{remark}[theorem]{Remark}
\newtheorem{proposition}[theorem]{Proposition}
\newtheorem{definition}[theorem]{Definition}
\newcolumntype{C}[1]{>{\centering\let\newline\\\arraybackslash\hspace{0pt}}m{#1}}

\usepackage[utf8]{inputenc}
\usepackage{amsfonts}
\usepackage{amssymb}
\usepackage{verbatim}
\usepackage{booktabs}

\numberwithin{equation}{section}

\begin{document}
\title{Quadratic codimension growth and minimal varieties of unitary superalgebras with superinvolution}

\author{Wesley Quaresma Cota$^{1,a,*}$ and Luiz Henrique de Souza Matos$^{2,b}$}
\thanks{{\it E-mail addresses:} quaresmawesley@gmail.com (Cota), luizmatos@ufmg.br (Matos).}

\thanks{\footnotesize $^{1}$ Partially supported by FAPESP. Process Number 2025/05699-0.}

\thanks{\footnotesize $^{2}$ Partially supported by CAPES and FAPEMIG}

\thanks{\footnotesize $^{*}$ Corresponding author}



\dedicatory{$^{a}$ IME, USP, Rua do Matão 1010, 05508-090, São Paulo, Brazil \\ $^b$ ICEx, UFMG, Avenida Antonio Carlos 6627, 31123-970, Belo Horizonte, Brazil}


\begin{abstract}
Let $A$ be a unitary associative superalgebra with superinvolution over a field of characteristic zero. In this paper, we classify unitary $*$-varieties whose $*$-codimension sequences have quadratic growth. We first establish a structural connection between the non-zero multiplicities of the proper $*$-cocharacters and a family of explicit finite-dimensional model algebras. Using this connection, we determine all minimal unitary $*$-varieties of quadratic codimension growth. We then prove that every unitary $*$-algebra with quadratic $*$-codimension growth is $T_*$-equivalent to a finite direct sum of these minimal model algebras. Consequently, the structure of such varieties is completely determined by the proper $*$-cocharacters occurring in degrees at most two.
\end{abstract}

\maketitle

\section{Introduction}

The study of polynomial identities of an algebra is a central topic in noncommutative algebra. However, determining the full set of polynomial identities satisfied by a given algebra is a highly nontrivial problem. To overcome this difficulty, Regev introduced in~\cite{RG} the sequence of codimensions $\{c_n(A)\}_{n\geq 1}$ of an algebra $A$, where $c_n(A)$ is the dimension of the space of multilinear polynomials of degree $n$ modulo the $T$-ideal of polynomial identities of $A$. As established by Kemer \cite{Kem1}, over a field of characteristic zero, every polynomial identity follows from finitely many multilinear ones. Consequently, the sequence of codimensions captures the asymptotic growth of the polynomial identities satisfied by~$A$. Regev~\cite{RG} proved that this sequence is exponentially bounded for PI-algebras and conjectured that its asymptotic behavior satisfies 
\[
c_n(A)\approx \alpha\, n^t d^n,
\]
for suitable constants $\alpha$, $t$, and $d$.

A fundamental advance toward this conjecture was obtained by Giambruno and Zaicev \cite{GiZai}, who proved that the limit
\[
\exp(A)=\lim_{n\to\infty}\sqrt[n]{c_n(A)}
\]
exists and is a non-negative integer, thus confirming Amitsur's conjecture. Moreover, Kemer \cite{Kem1} showed that $\{c_n(A)\}_{n\geq 1}$ is polynomially bounded if and only if the variety generated by $A$ does not contain the infinite-dimensional Grassmann algebra $\mathcal{G}$ and the algebra $UT_2$ of $2\times 2$ upper triangular matrices. Therefore, $\mathrm{var}(\mathcal{G})$ and $\mathrm{var}(UT_2)$ are the only varieties of almost polynomial growth, that is, they have exponential codimension growth, whereas every proper subvariety of either of them has polynomial codimension growth.

Varieties of polynomial growth have since been extensively investigated (see, for instance, \cite{Petro, GiamLaMa, Mara3}). The study of polynomial codimension growth has also been extended to algebras equipped with additional structures, such as gradings, involutions, and superinvolutions (\cite{Ioppololamat, Plamen, LaMaMisso, Maralice}). In these settings, classifying varieties beyond linear growth involves substantial combinatorial difficulties. An approach to simplify the problem is to restrict the investigation to unitary algebras, which makes it feasible to study varieties of quadratic growth (\cite{Mallu, Wesley1, Wesley2}). Even in this setting, obtaining a complete structural characterization by explicitly exhibiting the generating algebras remains a difficult problem. The present work addresses this problem in the setting of associative algebras endowed with a superinvolution, also called $*$-algebras.


In the context of $*$-algebras satisfying an ordinary identity, the sequence of $*$-codimensions $\{c_n^*(A)\}_{n \ge 1}$ either grows exponentially or is polynomially bounded (see, for instance, \cite{GIL2}). Consequently, recent works have focused on the characterization of varieties with polynomial $*$-codimension growth. In this case, it was shown in \cite{GIL} that a $*$-algebra has polynomial codimension growth if and only if its variety does not contain three specific algebras generating varieties of almost polynomial growth. Subsequently, the full classification of varieties with linear $*$-codimension growth was achieved in \cite{Ioppololamat}.

A variety $\mathcal V$ is said to be minimal of polynomial growth $n^k$ if $c_n^*(\mathcal V)\approx q\, n^k$ for $q>0$ and $k\ge1$, and every proper subvariety has polynomial codimension growth $n^t$ for some $t<k$. Minimal varieties play a key structural role. In fact, in many classification results, algebras generating minimal varieties appear as building blocks for the construction of the algebras generating varieties of polynomial growth (see, for instance, \cite{Dafne, Mallu, Wesley1, Wesley2, Petro, Mara3}). Thus, identifying all minimal varieties and understanding their aspects is a crucial step toward the complete classification of varieties of polynomial growth.


In this paper, to address the classification of varieties with quadratic growth, we adopt a structural approach that avoids exhaustive computational verifications. One of the main contributions of this paper is the establishment of an explicit connection between patterns of proper cocharacters and suitable generating algebras. This mapping provides a way to identify the generators of a variety directly from its cocharacter sequence. Using this correspondence, we provide a complete structural classification of unitary $*$-varieties with quadratic codimension growth. As a consequence, we prove that any unitary $*$-algebra with quadratic growth is $T_*$-equivalent to a finite direct sum of these minimal algebras.

\section{On unitary $*$-algebras}

In this paper, we work with associative superalgebras endowed with a superinvolution. Throughout this paper, $F$ denotes a field of characteristic zero and $A$ is an associative algebra over $F$.

Recall that $A$ is a superalgebra if it admits a decomposition $A = A_0 \oplus A_1$ as a direct sum of subspaces $A_0$ and $A_1$, called the homogeneous components of degrees $0$ and $1$, respectively, such that
\[
A_0A_0 + A_1A_1 \subseteq A_0
\quad \text{and} \quad
A_0A_1 + A_1A_0 \subseteq A_1.
\]

With this structure, $A$ becomes an algebra graded by the cyclic group $\mathbb{Z}_2$. In order to emphasize the corresponding $\mathbb{Z}_2$-grading, we shall denote a superalgebra $A$ by $A = (A_0, A_1).$ Moreover, for a non-zero element $a \in A_i$, $i\in \{0,1\}$, we denote by $|a|=i$ the homogeneous degree of $a$.

\begin{example} 
Let $M_n(F)$ be the algebra of $n\times n$ matrices over $F$ and let $\mathsf{g}=(g_1,\ldots,g_n)\in\mathbb{Z}_2^n$ be an arbitrary $n$-tuple with entries in $\mathbb{Z}_2$. The tuple $\mathsf{g}$ induces a $\mathbb{Z}_2$-grading on $M_n(F)$ defined by
\[
(M_n(F))_{0}=\textnormal{span}_F\{e_{ij}\mid g_i+g_j=0\}
\quad\text{and}\quad
(M_n(F))_{1}=\textnormal{span}_F\{e_{ij}\mid g_i+g_j=1\}.
\]
This grading is called the {elementary $\mathbb{Z}_2$-grading} induced by $\mathsf{g}$. As will be discussed below, this grading can also be restricted to certain subalgebras of $M_n(F)$.
\end{example}

Let $A=A_0\oplus A_1$ be a superalgebra. A {superinvolution} on $A$ is a linear map
$* \colon A \to A$ satisfying
\[A_i^*\subseteq A_i,\quad 
(a^*)^* = a
\quad \text{and} \quad
(ab)^* = (-1)^{|a||b|} \, b^* a^*,
\]
for all homogeneous elements $a,b \in A_0 \cup A_1$ and $i\in\{0,1\}$.

\begin{definition}
A superalgebra $A$ endowed with a superinvolution $*$ will be called an algebra with superinvolution or simply a $*$-algebra.    
\end{definition}

\begin{example}
Let $UT_n$ be the algebra of $n\times n$ upper triangular matrices over $F$. Consider the elementary $\mathbb Z_2$-grading on $UT_n$ induced by an $n$-tuple $\mathsf{g}=(g_1,\ldots,g_n)\in\mathbb Z_2^n$
such that $g_i+g_{n-i+1}=g_j+g_{n-j+1}$ for all $1\leq i,j\leq n$. Assume, in addition, that $(UT_n)_1^2={0}.$
Define the linear map $\rho:UT_n\to UT_n$ by reflection across the secondary diagonal: $$e_{ij}^{\rho}=e_{n-j+1,n-i+1},
\qquad 1\leq i\leq j\leq n.$$
The condition on $\mathsf{g}$ ensures that $\rho$ preserves the $\mathbb Z_2$-grading. Moreover, since $(UT_n)_1^2={0}$, for all homogeneous elements $a,b\in UT_n$ we have $(ab)^\rho=(-1)^{|a||b|}b^\rho a^\rho.$ Therefore, $\rho$ is a superinvolution on $UT_n$.
\end{example}

It is well known that, for a $*$-algebra $A = A_0 \oplus A_1$, one has the decomposition
\[
A = A_0^+ \oplus A_1^+ \oplus A_0^- \oplus A_1^-,
\]
where, for each $i \in \{0,1\}$, $A_i^+ = \{\, a \in A_i \mid a^* = a \,\}$ and $A_i^- = \{\, a \in A_i \mid a^* = -a \,\}$ are, respectively, the homogeneous symmetric and homogeneous skew components of degree $i$.

Let $X=\{x_1,x_2,\ldots\}$ be a countable set of variables and let $F\langle X\rangle$ be the free associative algebra generated by $X$ over $F$. We endow $F\langle X\rangle$ with a structure of $*$-algebra by writing $X$ as the disjoint union of the sets $X_i^\varepsilon=\{x_{1,i}^\varepsilon,x_{2,i}^\varepsilon,\ldots\}$, where $i\in \{0,1\}$ and $\varepsilon\in\{+,-\}$. Let $\mathcal{F}
=
F\langle X_0^+ \cup X_0^- \cup X_1^+ \cup X_1^-, * \rangle$ be the free associative algebra generated by
$X_0^+ \cup X_0^- \cup X_1^+ \cup X_1^-$ over $F$. 

A natural $\mathbb{Z}_2$-grading on $\mathcal{F}$ is defined by declaring that a monomial is homogeneous of degree $0$ if it contains an even number of variables from $X_1^+ \cup X_1^-$, and homogeneous of degree $1$ otherwise.  Moreover, we declare that the variables in $X_0^+ \cup X_1^+$ are symmetric, while those in $X_0^- \cup X_1^-$ are skew. In this way, $\mathcal{F}$ becomes a $*$-algebra, called the {free associative algebra with superinvolution}. An element $f\in \mathcal{F}$ is called a $*$-polynomial or simply a polynomial.

 \begin{definition}  A polynomial $f \in \mathcal{F}$ is called a {$*$-identity} of a $*$-algebra $A$ if $\lambda(f)=0$
for every evaluation $\lambda$ of the variables in $f$ by elements of $A$ such that $\lambda(x_{j,i}^\varepsilon)\in A_i^\varepsilon,$ for all $
i\in\{0,1\}$ and $ \varepsilon\in\{+,-\}.$ In this case, we write $f \equiv 0$ on $A$.\end{definition}

Denote by
\[
\Idd(A)=\{\, f\in \mathcal{F}\mid f\equiv 0 \text{ on } A \,\}
\]
the set of all $*$-identities of $A$. It is well known that $\Idd(A)$ is a $T_*$-ideal, that is, an ideal of $\mathcal{F}$ invariant under all endomorphisms of $\mathcal{F}$ that preserve the $\mathbb{Z}_2$-grading and commute with the superinvolution $*$. 

Moreover, since the base field $F$ is of characteristic zero, every $*$-identity is a consequence of the multilinear $*$-identities. Hence, we denote by
\[
P_n^*
=
\operatorname{span}_F
\bigl\{
w_{\sigma(1)} \cdots w_{\sigma(n)}
\ \big| \
\sigma \in S_n,\ 
w_j \in \{x_{j,i}^+, x_{j,i}^-\},\ 
i \in \{0,1\}
\bigr\}
\]
the vector space of multilinear $*$-polynomials of degree $n$.

In order to capture the asymptotic behavior of the growth of the polynomial identities of $A$, we define the $n$-th $*$-codimension of $A$ as
\[
c_n^*(A)
=
\dim_F
\frac{P_n^*}{P_n^* \cap \Idd(A)},
\qquad n \geq 1.
\]

The sequence of $*$-codimensions has been extensively studied (see, for instance, \cite{GIL2, GIL, Ioppololamat}). One of the main results in this direction was obtained in \cite{GIL2} and is stated below.

\begin{proposition}
   Let \(A\) be a PI-algebra with superinvolution \( * \). Then the sequence of
\( * \)-codimensions \(c_n^*(A)\), \(n\ge 1\), is exponentially bounded.
\end{proposition}

Moreover, in \cite{GIL}, the authors showed that the sequence $c_n^*(A),\,  n\geq 1$, either grows exponentially or is polynomially bounded. In the latter case, we say that $A$ has polynomial growth of the sequence of codimensions. 

In the study of polynomial identities, it is convenient to introduce the notion of variety. 
Let $A$ be a $*$-algebra and denote by $\mathcal{V}=\operatorname{var}^*(A)$ the $*$-variety generated by $A$, that is, the class of all $*$-algebras $B$ such that $\operatorname{Id}^*(A)\subseteq \operatorname{Id}^*(B).$ If the variety $\mathcal{V}$ is generated by a unitary $*$-algebra $A$, then both $A$ and $\mathcal{V}$ are said to be unitary.

It follows immediately from the definition that if $B\in \operatorname{var}^*(A)$, then
\[
c_n^*(B)\leq c_n^*(A),
\quad \text{for all } n\geq 1.
\]

Moreover, we say that two $*$-algebras $A$ and $B$ are {$T_*$-equivalent} if they satisfy the same $*$-polynomial identities, that is, $\operatorname{Id}^*(A)=\operatorname{Id}^*(B).$

 In \cite[Theorem 2.4]{Ioppololamat}, the authors characterized the varieties of polynomial codimension growth in terms of an algebra generating the variety.

\begin{theorem} \label{f+j}
Let $A$ be a $*$-algebra over a field $F$ of characteristic zero. Then $A$ has polynomial codimension growth if and only if $A \sim_{T_*} B$, where $B=B_1 \oplus \dots \oplus B_m$ and $B_1, \dots, B_m$ are finite-dimensional $*$-algebras over $F$ with $\dim_F B_i/J(B_i) \leq 1$, for all $i=1, \dots, m$.
\end{theorem}

In this paper, we are interested in $*$-algebras with polynomial codimension growth. In light of the previous result, $*$-algebras of type $B = F + J(B)$ are of particular interest. According to \cite[Theorem 4.1]{GIL2}, the Jacobson radical $J(B)$ of a $*$-algebra $B$ is a $*$-ideal. Consequently, $J(B)$ admits the subspace decomposition
\[
J(B) = J(B)_{0}^{+} + J(B)_{0}^{-}+ J(B)_{1}^{+} + J(B)_{1}^{-},
\]
where $J(B)_i^\epsilon$ denotes the homogeneous symmetric or skew component (for $\epsilon=+$ or $\epsilon=-$, respectively) of degree $i\in \{0,1\}$. This decomposition will be used repeatedly throughout this paper.
\black

Recall that, for \(y,z\in \mathcal{F}\), the commutator of length two is defined by $[y,z]=yz-zy.$ More generally, the commutator of length \(n\) is defined inductively by $[y_1,\ldots,y_n]=\bigl[[y_1,\ldots,y_{n-1}],\,y_n\bigr],$ for all $y_i\in \mathcal{F}.$

A polynomial $f\in P_n^*$ is called a proper $*$-polynomial if $f$ is a linear combination of monomials of type
\[
x_{i_1,0}^-\cdots x_{i_r,0}^- x^+_{j_1,1}\cdots x^+_{j_s,1}
x^-_{l_1,1}\cdots x^-_{l_t,1}
w_1\cdots w_v,
\]
where $w_1,\dots,w_v$ are commutators in the variables from $X_0^+\cup X_0^-\cup X_1^+\cup X_1^-$.

In the context of unitary $*$-algebras, working with proper $*$-polynomials is essential. Indeed, for a unitary $*$-algebra we have the following.

\begin{proposition} \cite{Willer}
    Let A be a unitary $*$-algebra over a field of characteristic zero. Then $\textnormal{Id}^*(A)$ is generated, as a $T_*$-ideal, by its multilinear proper $*$-polynomials.
\end{proposition}

Denote by $\Gamma_n^*$ the space of proper multilinear $*$-polynomials of degree $n$ and we set $\Gamma^{*}_0=\operatorname{span}_F\{1\}$.
The sequence given by
\[
\gamma^{*}_n(A)=
\dim_F\frac{\Gamma^{*}_n}{\Gamma^{*}_n\cap \text{Id}^*(A)},\qquad n\geq 1,
\]
is called the sequence of proper $*$-codimensions of $A$. This sequence carries all the information regarding the asymptotic behavior of the polynomial identities of $A$. The relationship between the $*$-codimensions and the proper $*$-codimensions of $A$ is given by
\begin{equation}\label{codimension}
c^{*}_n(A)=\sum_{i=0}^n \binom{n}{i}\,\gamma^{*}_i(A),\qquad n\ge 1,
\end{equation}
as can be found in \cite{Drensky}.

We recall that $c_n^*(A)\approx \alpha n^k$ means that $\lim_{n\to\infty}\frac{c_n^*(A)}{\alpha n^k}=1.$ Using this notation, we have the following.

\begin{proposition} \cite{Willer} \label{Gammancontidoemid}
Let \(A\) be a unitary \( * \)-algebra over a field \(F\) of characteristic zero.
Then, $\Gamma_{k+i}^* \subseteq \langle \Gamma_k^* \rangle_{T_*},$  for every \(k\geq 1\) and all \(i\geq 0\). Moreover, if \(A\) has polynomial codimension growth, then we have the following:
\begin{enumerate}
    \item[(1)]
    There exists an integer \(k\geq 0\) such that
    \(\Gamma_{k+1}^* \subseteq \textnormal{Id}^*(A)\).
    Moreover, \(\gamma_k^*(A)\neq 0\) and
    \(\gamma_{k+i}^*(A)=0\), for all \(i>0\). Consequently, $c_n^*(A)\approx \alpha n^k$, for some $\alpha \neq 0$.

    \item[(2)]
    The sequence of \( * \)-codimensions satisfies
    \[
    c_n^{*}(A)=\sum_{i=0}^k \binom{n}{i}\,\gamma_i^{*}(A),
    \]
    and hence \(c_n^{*}(A)\) is a polynomial in \(n\) with rational
    coefficients.
\end{enumerate}
\end{proposition}

  A useful approach to studying proper $*$-polynomials is to consider the multihomogeneous components of the vector space $\Gamma_n^*$. Let $n = n_1 + n_2 + n_3 + n_4$ be a sum of four non-negative integers. We denote by $\Gamma_{n_1, \ldots, n_4}$ the subspace of $\Gamma_n^*$ spanned by the multilinear proper polynomials where the first $n_1$ variables are symmetric  of homogeneous degree $0$, the second $n_2$ variables are skew of homogeneous degree $0$, the third $n_3$ variables are symmetric of homogeneous degree $1$ and the last $n_4$ variables are skew of homogeneous degree $1$.

By means of a combinatorial argument, we obtain the following relationship between $\Gamma_n^*$ and $\Gamma_{n_1,\ldots,n_4}$.

\begin{equation}\label{Gamma_n}
    \Gamma_n^{*}\cong \bigoplus_{n_1+\dots+n_4=n} \binom{n}{n_1, \ldots, n_{4}}\Gamma_{{n_1}, \ldots , {n_{4}}}.
\end{equation}

For each decomposition $n = n_1 + \cdots + n_4$, we define
\[
\Gamma_{n_1,\ldots,n_4}(A)
=
\frac{\Gamma_{n_1,\ldots,n_4}}{\Gamma_{n_1,\ldots,n_4} \cap \mathrm{Id}^*(A)},
\]
and set $\gamma_{n_1,\ldots,n_4}(A)
=
\dim_F \Gamma_{n_1,\ldots,n_4}(A)$, which is called the proper $(n_1,\ldots,n_4)$-codimension. Consequently, the sequence of $*$-codimensions can be related to the proper $(n_1,\ldots,n_4)$-codimensions through the following equation.

\begin{equation}\label{proper_codimension}
    \gamma_n^{*}(A)=\sum _{n_1+\cdots+n_{4}=n} \binom{n}{n_1,\ldots,n_{4}}\gamma_{n_1,\ldots,n_{4}}(A).
\end{equation}

It is well known that there exists a natural left action of the product of symmetric groups $S_{n_1,\ldots, n_4} := S_{n_1} \times S_{n_2} \times S_{n_3} \times S_{n_{4}}$ on $\Gamma_{n_1, \ldots,n_{4}}$, where each group $S_{n_i}$ acts by permuting the respective variables associated to $n_i$, for $i \in \{1,\ldots,4\}$. We denote by $\chi(\Gamma_{n_1,\ldots,n_{4}})$ the $S_{ n_1, \ldots ,n_4 }$-character of $\Gamma_{n_1,\ldots,n_{4}}$.

Observe that $\Gamma_{n_1,\ldots,n_4} \cap \mathrm{Id}^*(A)$ is invariant under this action. Hence,
$\Gamma_{n_1,\ldots,n_4}(A)$ is a left $S_{n_1,\ldots,n_4}$-module.
Let $\pi_{n_1,\ldots,n_4}(A)$ denote its $S_{n_1,\ldots,n_4}$-character, called the proper
$(n_1,\ldots,n_4)$-cocharacter of $\Gamma_{n_1,\ldots,n_4}(A)$. By complete reducibility, the character $\pi_{n_1,\ldots,n_4}(A)$ decomposes as
\begin{equation}\label{cocharacter}
\pi_{n_1,\ldots,n_4}(A)
=
\sum_{(\lambda_1,\ldots,\lambda_4)\vdash (n_1,\ldots,n_4)}
m_{\lambda_1,\ldots,\lambda_4}\,
\chi_{\lambda_1}\otimes \chi_{\lambda_2}\otimes \chi_{\lambda_3}\otimes\chi_{\lambda_4},
\end{equation}
where $(\lambda_1,\ldots,\lambda_4)$ is a multipartition of $(n_1,\ldots,n_4)$, that is,
$\lambda_i \vdash n_i$ for all $i=1,\ldots,4$, and
$\chi_{\lambda_1}\otimes\cdots\otimes\chi_{\lambda_4}$ is the irreducible
$S_{n_1,\ldots,n_4}$-character associated with this multipartition.

The degree of the irreducible character
$\chi_{\lambda_1}\otimes\cdots\otimes\chi_{\lambda_4}$ is given by
$d_{\lambda_1}d_{\lambda_2}d_{\lambda_3} d_{\lambda_4}$, where $d_{\lambda_i}$ denotes the degree of the
irreducible $S_{n_i}$-character $\chi_{\lambda_i}$, computed via the Hook Formula
\cite[Theorem~3.10.2]{Sagan}.

A standard approach to determine the multiplicities $m_{\lambda_1,\ldots,\lambda_4}$ relies on the representation theory of the general linear group $GL_m$. In what follows, we take for granted the basic facts concerning the structure of irreducible $GL_m\times GL_m\times GL_m\times GL_m $-modules, in particular the description of their generators via highest weight vectors. Since these notions are well known, we do not reproduce the theory here. A more detailed and comprehensive treatment can be found in \cite{Drensky} and also \cite{Mallu} for this theory applied to the space of proper polynomials in the setting of $*$-algebras.

 In \cite[Section 3]{Mallu}, the authors exhibit the construction of proper highest weight vectors (proper h.w.v.'s) $f_{\lambda_1,\ldots,\lambda_{4}}$ associated to a multipartition $(\lambda_1, \ldots,\lambda_{4}) \vdash (n_1,\ldots,n_{4})$, where $n=n_1+\cdots +n_4$ and $n\in \{1,2\}$, in the case of superalgebras with graded involution. Since this theory can be applied to our context, we shall make use of these results.

In the following, we present the main result on the computation of the multiplicities $m_{\lambda_1,\ldots,\lambda_{4}}$.

\begin{proposition}\label{prop:non-zero_multiplicities} \cite{Nascimento}
    Let $A$ be a unitary $*$-algebra with proper $(n_1, \ldots, n_4)$-cocharacter as in (\ref{cocharacter}). Then $m_{\lambda_1,\ldots,\lambda_{4}}\neq 0$ if and only if there exists a proper h.w.v. $f_{\lambda_1,\ldots,\lambda_{4}}$ associated to $(\lambda_1,\ldots,\lambda_{4}) \vdash (n_1, \ldots, n_4)$ such that $f_{\lambda_1,\ldots,\lambda_{4}} \notin \text{Id}^*(A)$. Moreover, $m_{\lambda_1,\ldots,\lambda_{4}}$ is equal to the maximal number of proper h.w.v.'s associated to $(\lambda_1,\ldots,\lambda_{4}) \vdash (n_1, \ldots, n_4)$ which are linearly independent modulo $\textnormal{Id}^*(A)$.
\end{proposition}

\begin{remark}\label{remark}
     Let $A$ be a unitary $*$-algebras with proper $(n_1, \ldots, n_4)$-cocharacter as in (\ref{cocharacter}) and let $B$ be a unitary $*$-algebra such that $B \in \varGstar(A)$ with $(n_1, \ldots, n_4)$-cocharacter given by
\begin{equation*}
		\pi_{n_1,\ldots,n_{4}}(B)=\underset{(\lambda_1,\ldots,\lambda_{4}) \vdash (n_1, \ldots, n_4)}{\sum} {\widetilde{m}}_{\lambda_1,\ldots,\lambda_{4}} \chi_{\lambda_1} \otimes \chi_{\lambda_2} \otimes \chi_{\lambda_3} \otimes \chi_{\lambda_{4}}.
	\end{equation*}
    Since $\Gamma_{n_1,\ldots,n_{4}} (B)$ can be embedded into $\Gamma_{n_1,\ldots,n_{4}} (A)$ for all $n=n_1+\cdots+ n_4 $, then we have $\widetilde{m}_{\lambda_1,\ldots,\lambda_{4}} \leq {m}_{\lambda_1,\ldots,\lambda_{4}} $ for every multipartition $(\lambda_1,\ldots,\lambda_{4}) \vdash (n_1, \ldots, n_4)$.
\end{remark}

 In order to simplify the notation, we write the $4$-tuple $(n_1,\ldots,n_4)$ as $\bigl({n_1}_{_{0^+}},{n_2}_{_{0^-}},{n_3}_{_{1^+}},{n_4}_{_{1^-}}\bigr)$, where zero components are omitted. Moreover, given a multipartition $(\lambda_1,\ldots,\lambda_4)\vdash (n_1,\ldots,n_4),$ we denote it by $\bigl((\lambda_1)_{0^+},(\lambda_2)_{0^-},(\lambda_3)_{1^+},(\lambda_4)_{1^-}\bigr),$ omitting the empty partitions. For instance, if \(n_1=2\), \(n_4=1\) and \(n_2=n_3=0\), then the tuple \((2,0,0,1)\) is written as
\((2_{0^+},1_{1^-})\) and the multipartition
\(\bigl((1,1),\emptyset,\emptyset,(1)\bigr)\vdash (2,0,0,1)\) is denoted by
\(\bigl((1,1)_{0^+},(1)_{1^-}\bigr)\).

Using the notation established above, we present the following table containing information about $\Gamma_{n_1, \ldots, n_4}$, where $n=n_1+ \cdots +n_4$ and $n\in \{1,2\}.$ In the following, we recall that $y\circ z$ denotes the Jordan product $yz+zy$. 
Moreover, $u$ denotes any element in $\{0,1\}$ and $t^\epsilon, s^{\delta}\in \{0^-,1^+,1^-\}$ with $ t^\epsilon\neq s^\delta$.

 \begin{longtable}{lllc}
\endfirsthead
\endhead
\endfoot
\endlastfoot
\toprule

\textbf{$\Gamma_{(n_1, \ldots, n_{4})}$} 
& \textbf{proper $(n_1, \ldots, n_4)$-cocharacters} 
& \textbf{proper h.w.v.'s }
& \textbf{multiplicity} \\ 
\midrule

$\Gamma_{(1_{1^+})}$
& $\chi_{((1)_{1^+})}$
& $x_{1,1}^+$
& $1$ \\[2mm]

$\Gamma_{(1_{u^-})}$
& $\chi_{((1)_{u^-})}$
& $x_{1,u}^-$
& $1$ \\[2mm]

$\Gamma_{(2_{0^+})}$
& $ \chi_{((1,1)_{0^+})}$
& $ [x_{1,0}^+,x_{2,0}^+]$
& $1$ \\[2mm]

$\Gamma_{(2_{t^\epsilon})}$
& $\chi_{((2)_{t^\epsilon})}+ \chi_{((1,1)_{t^\epsilon})}$
& $x_{1,t}^\epsilon \circ x_{2,t}^\epsilon, \; [x_{1,t}^\epsilon,x_{2,t}^\epsilon]$
& $1,\, 1$ \\[2mm]

$\Gamma_{(1_{0^+},1_{t^\epsilon})}$
& $\chi_{((1)_{0^+})} \otimes \chi_{((1)_{t^\epsilon})}$
& $[x^+_{1,0},x^\epsilon_{2,t}]$
& $1$ \\[2mm]

$\Gamma_{(1_{s^\delta},1_{t^\epsilon})}$
& $2\chi_{((1)_{s^\delta})}\otimes \chi_{((1)_{t^\epsilon})}$
& $[x^\delta_{1,s},x^\epsilon_{2,t}],\; x^\delta_{1,s}x^\epsilon_{2,t}$
& $2$ \\ [2mm]

\toprule

\caption{Proper $(n_1,\ldots,n_{4})$-cocharacters of $\Gamma_{(n_1,\ldots,n_4)}$}
\label{tabela_mult}

\end{longtable}

\section{Minimal unitary $*$-varieties with quadratic codimension growth}

In this section, we introduce a family of unitary $*$-algebras with quadratic codimension growth. For each such $\ast$-algebra, we explicitly describe its codimension sequence and its proper $(n_1,\ldots,n_4)$-cocharacters. As a consequence, we prove that the corresponding $*$-algebras generate minimal varieties of quadratic codimension growth. 

We now recall the definition of minimal varieties in the setting of $*$-algebras. 

\begin{definition}
A variety $\mathcal V$ is called {minimal of polynomial growth $n^k$} if $c_n^*(\mathcal V)\approx \alpha n^k$, for some constant $\alpha \neq 0$, and every proper subvariety $\mathcal{W}\subsetneq \mathcal V$ has polynomial codimension growth $c_n^*(\mathcal W)\approx \beta n^t $ for some $ t<k$ and $\beta \neq 0$. 
\end{definition}

For \(k \geq 2\), let \(UT_k\) denote the algebra of \(k \times k\) upper triangular matrices over \(F\). For \(m \geq 2\), consider the following elements of \(UT_{2m}\):
 $$E = \sum\limits_{i = 2}^{m-1} e_{i,i+1} + e_{2m-i,2m-i+1} \quad \mbox{ and }\quad I_{2m}=\sum_{i=1}^{2m} e_{ii}.$$ We define the following subalgebras of \(UT_{2m}\):     
 
 $$N_m = \operatorname{span}_F \{I_{2m}, E, \ldots , E^{m-2}; e_{12} - e_{2m-1,2m}, e_{13}, \ldots , e_{1m}, e_{m+1,2m}, e_{m+2,2m}, \ldots , e_{2m-2,2m}\},$$
        $$ U_m = \operatorname{span}_F \{ I_{2m}, E, \ldots , E^{m-2}; e_{12} + e_{2m-1,2m}, e_{13}, \ldots , e_{1m}, e_{m+1,2m}, e_{m+2,2m}, \ldots , e_{2m-2,2m}\}.$$ \\
   Denote by $N_{3,*}$ and $U_{3,*}$ the algebras $N_3$ and $U_3$, respectively, with trivial grading and  reflection superinvolution. Moreover, we denote by $N_{3,*}^{gri}$ and $U_{3,*}^{gri}$ the algebras $N_3$ and $U_3$, respectively, with the elementary grading induced by $(0,1,1,0,0,1)$ and the reflection superinvolution.

   In this work, we denote by $x_{i,j}$ any variable in the set $\{x_{i,j}^+, x_{i,j}^-\}$. The next lemma provides some information about the previous algebras.

\begin{lemma}[\cite{Ioppololamat}]\label{T-ideal-U_and_N}~For the $*$-algebras $N_{3,*}, U_{3,*}, N_{3,*}^{gri}$ and $U_{3,*}^{gri}$ we have
    \begin{enumerate}
        \item[1)] $\textnormal{Id}^{*}(N_{3,*}^{gri})=\langle  x_{1,0}^-, x_{1,1}x_{2,1}, [x_{1,0}^+,x_{2,1}^+] \rangle_{T_*}$ and    $\textnormal{Id}^{*}(U_{3,*}^{gri})=\langle  x_{1,0}^-, x_{1,1}x_{2,1}, [x_{1,0}^+,x_{2,1}^-] \rangle_{T_*}$.
        
        \item[2)]  $c_n^{*}(U_{3,*})=1+n+\displaystyle\binom{n}{2}$, $c_n^{*}(N_{3,*})=1+n+2\displaystyle\binom{n}{2}$ and $c_n^{*}(N_{3,*}^{gri})=c_n^{*}(U_{3,*}^{gri})=1+2n+2\displaystyle\binom{n}{2}$.

        \item[3)] The proper non-zero $(n_1, \ldots, n_{4})$-cocharacters of these algebras are given by Table \ref{tabela_U3-N3}.

  \begin{longtable}{ll}
\endfirsthead
\endhead
\endfoot
\endlastfoot
\toprule
$N_{3,*}$
& $\chi_{((1)_{0^-})},\quad  \chi_{((1)_{0^+})}\otimes \chi_{((1)_{0^-})}$ \\[2mm]

$U_{3,*}$
& $\chi_{((1)_{0^-})}, \quad \chi_{((1,1)_{0^+})}$ \\[2mm]

$N_{3,*}^{gri}$
& $\chi_{((1)_{1^+})}, \quad \chi_{((1)_{1^-})}, \quad  \chi_{((1)_{0^+})}\otimes \chi_{((1)_{1^-})}$ \\[2mm]

$U_{3,*}^{gri}$
& $\chi_{((1)_{1^+})}, \quad \chi_{((1)_{1^-})}, \quad  \chi_{((1)_{0^+})}\otimes \chi_{((1)_{1^+})}$ \\[2mm]

\toprule

\caption{Proper non-zero cocharacters of $N_3$ and $U_{3}$.}
\label{tabela_U3-N3}

\end{longtable}

    \end{enumerate}
\end{lemma}

\begin{proof}
The items 1) and 2) were established in \cite[Theorems 4.4 and 4.5]{Ioppololamat}.

To prove item 3), consider the $*$-algebra $N_{3,*}^{gri}$. Since $N_{3,*}^{gri}$ is a unitary $*$-algebra with quadratic codimension growth, Proposition~\ref{Gammancontidoemid} implies that $\Gamma_n^* \subseteq \textnormal{Id}^*(N_{3,*}^{gri})$, for all $n \geq 3.$ Moreover, since $x_{1,1}^+$ and $x_{1,1}^-$ do not belong to $\textnormal{Id}^*(N_{3,*}^{gri})$, it follows from Proposition~\ref{prop:non-zero_multiplicities} that the multiplicities
$m_{((1)_{1^+})}$ and $m_{((1)_{1^-})}$ are non-zero.

Consider the space $\Gamma_n^*$ and the corresponding proper highest weight vectors associated to the multipartitions
$(\lambda_1, \ldots, \lambda_4) \vdash (n_1, \ldots, n_4)$, for $2=n_1+\cdots +n_4$.
Observe that $[x_{1,0}^+, x_{1,1}^-] \notin \textnormal{Id}^*(N_{3,*}^{gri}).$ Hence, by Proposition~\ref{prop:non-zero_multiplicities}, we obtain $m_{((1)_{0^+}, (1)_{1^-})} \neq 0.$

Consequently, using relations~\eqref{codimension}, \eqref{proper_codimension} and~\eqref{cocharacter}, we obtain the lower bound
\[
1 + 2n + 2\binom{n}{2} \leq c_n^*(N_{3,*}^{gri}).
\]
Since $c_n^*(N_{3,*}^{gri}) = 1 + 2n + 2\binom{n}{2},$ it follows that the only non-zero proper $(n_1, \ldots, n_4)$-cocharacters of $N_{3,*}^{gri}$ are
\[
\chi_{((1)_{1^+})}, \quad
\chi_{((1)_{1^-})} \quad \text{and} \quad
\chi_{((1)_{0^+}, (1)_{1^-})}.
\]

The remaining cases follow analogously.
\end{proof}

Throughout this section, several tables of this form are presented. In each case, the proof is based on the same argument: the corresponding proper highest weight vector is exhibited and shown not to be an identity of the $*$-algebra. Therefore, the proofs will be omitted.

We now define $\mathcal{L}_1$ as the set of the previously defined $*$-algebras, i.e.,
\[
\mathcal{L}_1 =
\{N_{3,*},\, U_{3,*},\, N_{3,*}^{\mathrm{gri}},\, U_{3,*}^{\mathrm{gri}}\}.
\]

In \cite[Corollary 7.2]{ISDSV}, the authors proved the following result.

\begin{proposition}
The varieties generated by the $*$-algebras $B \in \mathcal{L}_1$ are minimal with quadratic codimension growth.
\end{proposition}

Consider the following commutative unitary algebras
\[
C_2=\left\{\begin{pmatrix}
a & b \\
0 & a
\end{pmatrix}\mid a,b\in F \right\}\quad \mbox{ and }\quad
C_3=\left\{\begin{pmatrix}
a & b & c \\
0 & a & b \\
0 & 0 & a
\end{pmatrix}\mid a,b,c\in F\right\}.
\]

We consider the linear maps $^{\mathrm{sup}}$ and $*$ on $C_2$ defined by:
\[
\begin{pmatrix}
a & b \\
0 & a
\end{pmatrix}^{\mathrm{sup}}
=
\begin{pmatrix}
a & b \\
0 & a
\end{pmatrix}\quad \mbox{ and }\quad 
\qquad
\begin{pmatrix}
a & b \\
0 & a
\end{pmatrix}^{*}
=
\begin{pmatrix}
a & -b \\
0 & a
\end{pmatrix}.
\]

Moreover, we define on $C_3$ the linear maps $i_1,i_2$ and $i_3$ given by:

\[
\begin{pmatrix}
a & b & c \\
0 & a & b \\
0 & 0 & a
\end{pmatrix}^{i_1}
=
\begin{pmatrix}
a & b & -c \\
0 & a & b \\
0 & 0 & a
\end{pmatrix},
\quad
\begin{pmatrix}
a & b & c \\
0 & a & b \\
0 & 0 & a
\end{pmatrix}^{i_2}
=
\begin{pmatrix}
a & -b & c \\
0 & a & -b \\
0 & 0 & a
\end{pmatrix},
\quad
\begin{pmatrix}
a & b & c \\
0 & a & b \\
0 & 0 & a
\end{pmatrix}^{i_3}
=
\begin{pmatrix}
a & -b & -c \\
0 & a & -b \\
0 & 0 & a
\end{pmatrix}.
\] \\

In what follows, we introduce a list of superalgebras endowed with a superinvolution.

\begin{itemize}
  \item[-] $C_{2,*}$ is the algebra $C_2$ with the trivial grading and superinvolution $*$.
  \item[-] $C_2^{\mathrm{gr}}$ is the algebra $C_2$ with $\mathbb{Z}_2$-grading
  $(F(e_{11}+e_{22}),\, F e_{12})$ and superinvolution $^{\mathrm{sup}}$.
  \item[-] $C_{2,*}^{\mathrm{gr}}$ is the algebra $C_2$ with $\mathbb{Z}_2$-grading
  $(F(e_{11}+e_{22}),\, F e_{12})$ and superinvolution $*$.
  \item[-] $C_{3,i_2}$ is the algebra $C_3$ with the trivial grading and superinvolution $i_2$.
  \item[-] $C_{3,i_1}^{\mathrm{gr}}$ is the algebra $C_3$ with $\mathbb{Z}_2$-grading
  $(F(e_{11}+e_{22}+e_{33}) + F e_{13},\, F(e_{12}+e_{23}))$ and superinvolution $i_1$.
  \item[-] $C_{3,i_3}^{\mathrm{gr}}$ is the algebra $C_3$ with $\mathbb{Z}_2$-grading
  $(F(e_{11}+e_{22}+e_{33}) + F e_{13},\, F(e_{12}+e_{23}))$ and superinvolution $i_3$.
\end{itemize}

\begin{lemma}\label{codim-C_2-C_3}~ For the $*$-algebras $C_{2,*}$, $C_2^{\mathrm{gr}}$, $C_{2,*}^{\mathrm{gr}}$, $C_{3,i_2}$, $C_{3,i_1}^{\mathrm{gr}}$ and $C_{3,i_3}^{\mathrm{gr}}$ we have
    \begin{enumerate}
        \item[1)]  $\textnormal{Id}^*(C_{3,i_1}^{\mathrm{gr}})=\langle [x_{1,0}^+,x], [x_{1,1}^+,x_{2,1}^+], x_{1,1}^-,x_{1,0}^-x_{2,0}^- ,x_{1,1}^+x_{2,0}^-, x_{1,1}^+x_{2,1}^+x_{3,1}^+ \rangle_{T_*}$.

        \item[2)]  $\textnormal{Id}^*(C_{3,i_3}^{\mathrm{gr}})=\langle [x_{1,0}^+,x], [x_{1,1}^-,x_{2,1}^-], x_{1,1}^+,x_{1,0}^-x_{2,0}^- ,x_{1,1}^-x_{2,0}^-, x_{1,1}^-x_{2,1}^-x_{3,1}^-\rangle_{T_*}$.
        
        \item[3)] $c_n^{*}(B_1)=1+n$, $c_n^{*}(B_2)=1+2n+\displaystyle\binom{n}{2}$ and $c_n^{*}(C_{3,i_2})=1+n+\displaystyle\binom{n}{2}$, for all $B_1\in \{C_{2,*}
,C_{2}^{gr}, C_{2,*}^{gr}\}$ and $B_2\in \{C_{3,i_1}^{gr}, C_{3,i_3}^{gr}\}$.

        \item[4)] The proper non-zero $(n_1, \ldots, n_{4})$-cocharacters of the algebras defined above are given in Table \ref{tabela_C2-c3}.

\begin{longtable}{ll}
\endfirsthead
\endhead
\endfoot
\endlastfoot
\toprule

$C_{2,*}$
& $\chi_{((1)_{0^-})}$ \\[2mm]

$C_{2}^{gr}$
& $\chi_{((1)_{1^+})}$\\[2mm]

$C_{2,*}^{gr}$
& $\chi_{((1)_{1^-})}$ \\[2mm]

$C_{3,i_2}$
& $\chi_{((1)_{0^-})},\quad  \chi_{((2)_{0^-})}$ \\[2mm]

$C_{3,i_1}^{gr}$
& $\chi_{((1)_{0^-})} , \quad \chi_{((1)_{1^+})}, \quad  \chi_{((2)_{1^+})}$ \\[2mm]

$C_{3,i_3}^{gr}$
& $\chi_{((1)_{0^-})},\quad \chi_{((1)_{1^-})}, \quad \chi_{((2)_{1^-})}$ \\[2mm]

\toprule
\caption{Proper non-zero cocharacters of $C_2$ and $C_3$.}
\label{tabela_C2-c3}
\end{longtable}
    \end{enumerate}
\end{lemma}
\begin{proof}
We start by proving item 1). Let $I$ be the $T_*$-ideal generated by the polynomials
\[
\langle [x_{1,0}^+,x],\; [x_{1,1}^+,x_{2,1}^+],\; x_{1,1}^-,\; x_{1,0}^-x_{2,0}^-,\;
x_{1,1}^+x_{2,0}^-,\; x_{1,1}^+x_{2,1}^+x_{3,1}^+ \rangle_{T_*}.
\]
It is straightforward to verify that $I \subseteq \textnormal{Id}^*(C_{3,i_1}^{\mathrm{gr}}).$

Let $f \in \textnormal{Id}^*(C_{3,i_1}^{\mathrm{gr}})$ be a multilinear polynomial of degree $n$. Since $I \subseteq \textnormal{Id}^*(C_{3,i_1}^{\mathrm{gr}})$, it follows that, modulo $I$, the polynomial $f$ can be written as a linear combination of the polynomials
\begin{equation} \label{baseofidentities}
\begin{aligned}
x_{1,0}^+ \cdots x_{n,0}^+,\quad  &\quad  x_{1,0}^+ \cdots \widehat{x_{i,0}^+} \cdots x_{n,0}^+\, x_{i,0}^-, \\
x_{1,0}^+ \cdots \widehat{x_{i,0}^+} \cdots x_{n,0}^+\, x_{i,1}^+\quad  &\mbox{ and } \quad x_{1,0}^+ \cdots \widehat{x_{p,0}^+} \cdots \widehat{x_{q,0}^+} \cdots x_{n,0}^+\,
x_{p,1}^+ x_{q,1}^+,
\end{aligned}
\end{equation}
where $p<q$  and the symbol $\widehat{x}$ denotes the omission of the variable $x$.

By the multihomogeneity of $T_*$-ideals, we may assume without loss of generality that $f$ is equivalent modulo $I$ to one of the following polynomials:
\[
\alpha\, x_{1,0}^+ \cdots x_{n,0}^+, \quad
\beta\, x_{1,0}^+ \cdots x_{n-1,0}^+ x_{n,0}^-, \quad
\sigma\, x_{1,0}^+ \cdots x_{n-1,0}^+ x_{n,1}^+, \quad
\delta\, x_{1,0}^+ \cdots x_{n-2,0}^+ x_{n-1,1}^+ x_{n,1}^+ .
\]

In each of the cases above, we consider the evaluation defined by
\[
x_{i,0}^+ \mapsto e_{11}+e_{22}+e_{33}, \quad
x_{n,0}^- \mapsto e_{13}, \quad
x_{n,1}^+ \mapsto e_{12}+e_{23}, \quad
x_{n-1,1}^+ \mapsto e_{12}+e_{23},
\]
for all $1 \leq i \leq n$.
Since $f$ is an identity of $C_{3,i_1}^{\mathrm{gr}}$, these evaluations yield $\alpha = \beta = \sigma = \delta = 0.$
Therefore, $f \in I$, and consequently, $I = \textnormal{Id}^*(C_{3,i_1}^{\mathrm{gr}}).$

Moreover, the argument above shows that no non-zero linear combination of the polynomials listed in~\eqref{baseofidentities} is an identity of $C_{3,i_1}^{\mathrm{gr}}$.
Hence, these polynomials form a basis of the vector space $P_n^*$ modulo $\textnormal{Id}^*(C_{3,i_1}^{\mathrm{gr}})$.
Therefore, it follows that
\[
c_n^{*}(C_{3,i_1}^{\mathrm{gr}}) = 1 + 2n + \binom{n}{2}.
\]

Finally, once the codimension is computed, an argument analogous to that used in Lemma~\ref{T-ideal-U_and_N} shows that $\{\chi_{((1)_{0^-})},\; \chi_{((1)_{1^+})},\; \chi_{((2)_{1^+})}\}$ is precisely the set of all non-zero proper $(n_1,\ldots,n_4)$-cocharacters of $C_{3,i_1}^{\mathrm{gr}}$.

A similar approach can be applied to prove the statements for the algebra $C_{3,i_3}^{\mathrm{gr}}$.

For the remaining algebras, the proofs can be found in \cite[Theorem 6.1]{Ioppololamat}, \cite[Lemma 9]{DMat} and \cite[Theorem 8.1]{LMFM}.    
\end{proof}

Consider now the following subalgebra of the infinite-dimensional Grassmann algebra:
\[
\mathcal{G}_2=\langle 1, e_1, e_2 \mid e_i e_j = -e_j e_i;\, i,j=1,2\rangle.
\]

 Now, we define the following $\mathbb{Z}_2$-gradings on $\mathcal{G}_2$:
\[ \mathcal{G}_2=(\mathcal{G}_2,0),\quad 
\mathcal{G}_2^{\mathrm{gr}}
=
(F1 + F e_1 e_2,\; F e_1 + F e_2) \quad \mbox{ and }\quad 
\mathcal{G}_2^{\mathrm{gri}}
=
(F1 + F e_1,\; F e_2 + F e_1 e_2).
\]

On the superalgebras defined above, we consider the superinvolutions:
\[
\psi(e_i)=e_i,\qquad 
\tau(e_i)=-e_i\quad \mbox{ and }\quad 
\gamma(e_i)=(-1)^i e_i,\qquad \mbox{ for } i=1,2.
\]

For each $* \in \{\psi,\tau,\gamma\}$, denote by $\mathcal{G}_{2,*}$, $\mathcal{G}_{2,*}^{gr}$ and $\mathcal{G}_{2,*}^{gri}$ the superalgebras $\mathcal{G}_{2}$, $\mathcal{G}_{2}^{gr}$ and $\mathcal{G}_{2}^{gri}$,  respectively, endowed with the superinvolution $*$.

In the following, we denote by $x$ any symmetric or skew variable of homogeneous degree $0$ or $1$.

\begin{lemma} \label{codim_G_2}~ For the algebras $\mathcal{G}_{2,*},\mathcal{G}_{2,*}^{{gr}} $ and $
\mathcal{G}_{2,*}^{\mathrm{gri}}$ we have
\begin{itemize}        
    \item[1)] $\textnormal{Id}^{*}(\mathcal{G}_{2,\tau}^{gr})=\langle x_{1,0}^-,x_{1,1}^+,[x_{1,0}^+,x],x_{1,1}^-\circ x_{2,1}^-,x_{1,1}^-x_{2,1}^-x_{3,1}^-  \rangle_{T_*}$.

   \item[2)]   $\textnormal{Id}^{*}(\mathcal{G}_{2,\psi}^{gr})=\langle  x_{1,0}^-,x_{1,1}^-,[x_{1,0}^+,x],x_{1,1}^+\circ x_{2,1}^+,x_{1,1}^+x_{2,1}^+x_{3,1}^+ \rangle_{T_*}$.

   \item[3)]   $\textnormal{Id}^{*}(\mathcal{G}_{2,\gamma}^{gr})=\langle x_{1,1}^-x_{2,1}^-,x_{1,1}^+x_{2,1}^+,x_{1,0}^-x_{2,0}^-,x_{1,1}^-x_{2,0}^-,x_{1,1}^+x_{2,0}^-, [x_{1,0}^+,x],x_{1,1}^-\circ x_{2,1}^+ \rangle_{T_*}$.

    \item[4)] $c_n^{*}(B_1)=1+n+\displaystyle\binom{n}{2}$, $c_n^{*}(\mathcal{G}_{2,\gamma}^{gr})=1+3n+2\displaystyle\binom{n}{2}$, $c_n^{*} (B_2)=1+2n+2\displaystyle\binom{n}{2}$, for all $B_1\in\{\mathcal{G}_{2,\tau}
, \mathcal{G}_{2,\psi}^{gr}, \mathcal{G}_{2,\tau}^{gr}\}$ and $B_2\in\{\mathcal{G}_{2,\tau}^{gri}, \mathcal{G}_{2,\gamma}^{gri}\}$.

    \item[5)] The proper non-zero $(n_1, \ldots, n_{4})$-cocharacters of these algebras are given by  Table \ref{tableG_2}.

\begin{longtable}{ll}
\endfirsthead
\endhead
\endfoot
\endlastfoot

\toprule 
$\mathcal{G}_{2,\tau}$
& $\chi_{((1)_{0^-})},\quad \chi_{((1,1)_{0^-})}$  \\[2mm]

$\mathcal{G}_{2,\psi}^{gr}$
&  $\chi_{((1)_{1^+})},\quad  \chi_{((1,1)_{1^+})}$ \\[2mm]

$\mathcal{G}_{2,\tau}^{gr}$
& $ \chi_{((1)_{1^-})}, \quad \chi_{((1,1)_{1^-})}$ \\[2mm]

$\mathcal{G}_{2,\gamma}^{gr}$
& $\chi_{((1)_{0^-})},\quad \chi_{((1)_{1^+})},\quad \chi_{((1)_{1^-})}, \quad \chi_{((1)_{1^+})} \otimes \chi_{((1)_{1^-})}$ \\ [2mm]

$\mathcal{G}_{2,\tau}^{gri}$
& $\chi_{((1)_{0^-})},\quad  \chi_{((1)_{1^-})},\quad \chi_{((1)_{0^-})}\otimes \chi_{((1)_{1^-})}$ \\[2mm]

$\mathcal{G}_{2,\gamma}^{gri}$
& $\chi_{((1)_{0^-})}, \quad \chi_{((1)_{1^+})}, \quad \chi_{((1)_{0^-})} \otimes \chi_{((1)_{1^+})}$ \\[2mm]
 \toprule

\caption{Proper non-zero cocharacters of $\mathcal{G}_{2,*},\mathcal{G}_{2,*}^{{gr}} $ and $\mathcal{G}_{2,*}^{{gri}}$.}
\label{tableG_2}

\end{longtable}
    \end{itemize}
\end{lemma} 

 \begin{proof}
    The proof for the $*$-algebras $\mathcal{G}_{2,\tau}^{gr},
\mathcal{G}_{2,\psi}^{gr}, \mathcal{G}_{2,\gamma}^{gr}$ follows as in the previous lemma. For the remaining $*$-algebras, the codimension can be found in \cite[Lemma 16]{LaMaMisso} and \cite[Lemma 7.3]{Nascimento}.
\end{proof} \black

Let $W$ be the commutative subalgebra of $UT_4$ given by

\[
W = \left\{
\begin{pmatrix}
a & b & c & d \\
0 & a & 0 & c \\
0 & 0 & a & b \\
0 & 0 & 0 & a
\end{pmatrix}
\ \mid \ a,b,c,d \in F
\right\}.
\] \\

We define the linear maps $\eta_1, \eta_2$ and $\eta_3$ on $W$ as follows:

\[
\begin{pmatrix}
a & b & c & d \\
0 & a & 0 & c \\
0 & 0 & a & b \\
0 & 0 & 0 & a
\end{pmatrix}^{\eta_1}
=
\begin{pmatrix}
a & -b & c & -d \\
0 & a & 0 & c \\
0 & 0 & a & -b \\
0 & 0 & 0 & a
\end{pmatrix}, \qquad \begin{pmatrix}
a & b & c & d \\
0 & a & 0 & c \\
0 & 0 & a & b \\
0 & 0 & 0 & a
\end{pmatrix}^{\eta_2}
=
\begin{pmatrix}
a & -b & c & d \\
0 & a & 0 & c \\
0 & 0 & a & -b \\
0 & 0 & 0 & a
\end{pmatrix},
\]
\[
\begin{pmatrix}
a & b & c & d \\
0 & a & 0 & c \\
0 & 0 & a & b \\
0 & 0 & 0 & a
\end{pmatrix}^{\eta_3}
=
\begin{pmatrix}
a & -b & -c & d \\
0 & a & 0 & -c \\
0 & 0 & a & -b \\
0 & 0 & 0 & a
\end{pmatrix}.
\] \\

In what follows, we introduce the following superalgebras with a superinvolution:

\begin{itemize}
    \item[-] $W_{\eta_2}^{\mathrm{gr}}$ is the algebra $W$ with $\mathbb{Z}_2$-grading $(F(e_{11}+\cdots+e_{44}) + F e_{14},\; F(e_{12}+e_{34}) + F(e_{13}+e_{24}))$ and superinvolution $\eta_2$.
    \item[-] $W_{\eta_1}^{\mathrm{gri}}$ is the algebra $W$ with $\mathbb{Z}_2$-grading $(F(e_{11}+\cdots+e_{44}) + F(e_{12}+e_{34}),\; F(e_{13}+e_{24}) + F e_{14})$ and superinvolution $\eta_1$.
    \item[-] $W_{\eta_3}^{\mathrm{gri}}$ is the algebra $W$ with $\mathbb{Z}_2$-grading $(F(e_{11}+\cdots+e_{44}) + F(e_{12}+e_{34}),\; F(e_{13}+e_{24}) + F e_{14})$ and superinvolution $\eta_3$.
\end{itemize}

\begin{lemma}\label{codim_W} ~ For the algebras $W_{\eta_2}^{\mathrm{gr}}$, $W_{\eta_1}^{\mathrm{gri}}$ and $W_{\eta_3}^{\mathrm{gri}}$ we have
    \begin{enumerate}
        \item[1)] $\textnormal{Id}^*(W_{\eta_2}^{gr})=\langle [x_{1,0}^+,x], [x_{1,1}^-,x_{2,1}^+], x_{1,0}^-,x_{1,1}^-x_{2,1}^- ,x_{1,1}^+x_{2,1}^+ \rangle_{T_*}$.

        \item[2)] $c_n^{*}(W_{\eta_2}^{gr})=1+2n+ 2\displaystyle\binom{n}{2}$ and $c_n^{*}(W_{\eta_1}^{gri})=c_n^{*}(W_{\eta_3}^{gri})=1+3n+ 2\displaystyle\binom{n}{2}$.
        
        \item[3)] The proper non-zero $(n_1, \ldots, n_{4})$-cocharacters of these algebras are given in Table \ref{table_W}.

\begin{longtable}{ll}
\endfirsthead
\endhead
\endfoot
\endlastfoot
\toprule

$W_{\eta_2}^{gr}$
& $\chi_{((1)_{1^+})},\quad  \chi_{((1)_{1^-})},\quad \chi_{((1)_{1^+})} \otimes \chi_{((1)_{1^-})}$ 

 \\[2mm]

$W_{\eta_3}^{gri}$
& $\chi_{((1)_{0^-})}, \quad  \chi_{((1)_{1^+})},\quad  \chi_{((1)_{1^-})},\quad \chi_{((1)_{1^-})} \otimes \chi_{((1)_{0^-})}$ \\[2mm]

$W_{\eta_1}^{gri}$
& $\chi_{((1)_{0^-})},\quad \chi_{((1)_{1^+})},\quad \chi_{((1)_{1^-})},\quad \chi_{((1)_{0^-})} \otimes \chi_{((1)_{1^+})}$ \\
[2mm]

\toprule
\caption{Proper non-zero cocharacters of $W_{\eta_2}^{\mathrm{gr}}$, $W_{\eta_1}^{\mathrm{gri}}$ and $W_{\eta_3}^{\mathrm{gri}}$.}
\label{table_W}
\end{longtable}
    \end{enumerate}
\end{lemma}

    \begin{proof} 
The $T_*$-ideal and the codimension of the $*$-algebras $W_{\eta_1}^{gri}$ and $W_{\eta_3}^{gri}$ were described in \cite[Lemma 4.4]{Mallu}. Thus, we focus only on the $*$-algebra $W_{\eta_2}^{gr}$. For item 1), note that 
$$I=\langle [x_{1,0}^+,x], [x_{1,1}^-,x_{2,1}^+], x_{1,0}^-,x_{1,1}^-x_{2,1}^- ,x_{1,1}^+x_{2,1}^+ \rangle_{T_*} \subseteq \textnormal{Id}^*(W_{\eta_2}^{gr}).$$ 
Moreover, it is easy to check that $$\Gamma_n^*= \Gamma_n^*\cap \textnormal{Id}^*(W_{\eta_2}^{gr}) \subseteq I,\mbox{ for all } n\geq 3.$$ Since $W_{\eta_2}^{gr}$ is a unitary $*$-algebra, it remains to analyze the multilinear proper identities of degrees $1$ and $2$.

Let $f$ be a multilinear proper identity of degree $1$ of
$W_{\eta_2}^{\mathrm{gr}}$. Observe that $x_{1,1}^+,\; x_{1,1}^- \notin \textnormal{Id}^*(W_{\eta_2}^{\mathrm{gr}})$, while
$x_{1,0}^- \in \textnormal{Id}^*(W_{\eta_2}^{\mathrm{gr}})\cap I.$ Therefore, $\Gamma_1^* \cap \textnormal{Id}^*(W_{\eta_2}^{\mathrm{gr}}) \subseteq I.$

Now let $f \in \Gamma_2^* \cap \textnormal{Id}^*(W_{\eta_2}^{\mathrm{gr}})$. Without loss of
generality, we may assume that $f$ is multihomogeneous. Reducing $f$ modulo $I$,
we may write $f=\alpha\, x_{1,1}^+x_{2,1}^-$, for some $ \alpha\in F.$ Considering the evaluation
\[
x_{1,1}^+ \mapsto e_{13}+e_{24},
\qquad
x_{2,1}^- \mapsto e_{12}+e_{34}.
\] we obtain $\alpha=0$. Consequently, $\Gamma_2^* \cap \textnormal{Id}^*(W_{\eta_2}^{\mathrm{gr}}) \subseteq I.$

Since $\textnormal{Id}^*(W_{\eta_2}^{\mathrm{gr}})$ is generated by its proper multilinear
polynomials, we conclude that $\textnormal{Id}^*(W_{\eta_2}^{\mathrm{gr}})=I.$

Finally, we analyze the proper highest weight vectors listed in
Table~\ref{tabela_mult}. From the description of the $T_*$-ideal of
$W_{\eta_2}^{\mathrm{gr}}$ obtained above, it follows that the only proper
highest weight vectors which are not identities of $W_{\eta_2}^{\mathrm{gr}}$ are $x_{1,1}^+,\, x_{1,1}^-$ and $x_{1,1}^+x_{2,1}^-.$ Therefore, the only non-zero proper cocharacters of
$W_{\eta_2}^{\mathrm{gr}}$ are $\chi_{((1)_{1^+})},$ $
\chi_{((1)_{1^-})}$ and $\chi_{((1)_{1^+})} \otimes \chi_{((1)_{1^-})}$.

The codimension sequence can now be computed using the relations~\eqref{codimension}, \eqref{proper_codimension}, and~\eqref{cocharacter}.

\end{proof}

In what follows, we present some results concerning the direct sum of the previously defined algebras.

\begin{lemma}\label{codim_G2_sum_W} ~ For the $*$-algebras $\mathcal{G}_{2,\gamma}^{gri} \oplus W_{\eta_1}^{gri}$, $\mathcal{G}_{2,\gamma}^{gr} \oplus W_{\eta_2}^{gr}$ and $\mathcal{G}_{2,\tau}^{gri} \oplus W_{\eta_3}^{gri}$ we have      \begin{enumerate}
        \item[1)] $\textnormal{Id}^*(\mathcal{G}_{2,\gamma}^{gri} \oplus W_{\eta_1}^{gri})=\langle [x_{1,0}^+,x], x_{1,0}^-x_{2,0}^-, x_{1,0}^-x_{2,1}^-, x_{1,1}^+x_{2,1}^+, x_{1,1}^+x_{2,1}^-, x_{1,1}^-x_{2,1}^- \rangle_{T_*}$.

        \item[2)] $\textnormal{Id}^*(\mathcal{G}_{2,\gamma}^{gr} \oplus W_{\eta_2}^{gr})=\langle [x_{1,0}^+,x], x_{1,1}^-x_{2,1}^-, x_{1,1}^-x_{2,0}^-, x_{1,1}^+x_{2,1}^+, x_{1,1}^+x_{2,0}^-, x_{1,0}^-x_{2,0}^- \rangle_{T_*}$.

        \item[3)] $\textnormal{Id}^*(\mathcal{G}_{2,\tau}^{gri} \oplus W_{\eta_3}^{gri})=\langle [x_{1,0}^+,x], x_{1,0}^-x_{2,0}^-, x_{1,0}^-x_{2,1}^+, x_{1,1}^-x_{2,1}^-, x_{1,1}^-x_{2,1}^+, x_{1,1}^+x_{2,1}^+ \rangle_{T_*}$.
        
        \item[4)] $c_n^{*}(B)=1+3n+
            {4}\displaystyle\binom{n}{2}$, where $B \in \{\mathcal{G}_{2,\gamma}^{gri} \oplus W_{\eta_1}^{gri},\mathcal{G}_{2,\gamma}^{gr} \oplus W_{\eta_2}^{gr}, \mathcal{G}_{2,\tau}^{gri} \oplus W_{\eta_3}^{gri} \}$.

        \item[5)] The proper non-zero $(n_1, \ldots, n_{4})$-cocharacters of these algebras are given by  Table \ref{table_sum}.
\begin{longtable}{ll}
\endfirsthead
\endhead
\endfoot
\endlastfoot
\toprule

$\mathcal{G}_{2,\gamma}^{gri} \oplus W_{\eta_1}^{gri}$
& $\chi_{((1)_{1^-})},\quad \chi_{((1)_{0^-})}, \quad \chi_{((1)_{1^+})}, \quad 2\chi_{((1)_{1^+})} \otimes \chi_{((1)_{0^-})}$ \\[2mm]

$\mathcal{G}_{2,\gamma}^{gr} \oplus W_{\eta_2}^{gr}$
& $\chi_{((1)_{1^-})},\quad \chi_{((1)_{0^-})}, \quad \chi_{((1)_{1^+})}, \quad 2\chi_{((1)_{1^+})} \otimes \chi_{((1)_{1^-})}$ \\[2mm]

$\mathcal{G}_{2,\tau}^{gri} \oplus W_{\eta_3}^{gri}$
& $\chi_{((1)_{1^-})},\quad  \chi_{((1)_{0^-})},\quad  \chi_{((1)_{1^+})},\quad 2\chi_{((1)_{0^-})} \otimes \chi_{((1)_{1^-})}$ \\[2mm]

\toprule

\caption{Proper non-zero cocharacters of $W$.}
\label{table_sum}

\end{longtable}

    \end{enumerate}
        
    \end{lemma}

\begin{proof} 
Since the proofs of the three cases are analogous, we focus on the $*$-algebra $\mathcal{G}_{2,\gamma}^{gri}\oplus W_{\eta_1}^{gri}$. We start by noting that 
$$I=\langle [x_{1,0}^+,x], x_{1,0}^-x_{2,0}^-, x_{1,0}^-x_{2,1}^-, x_{1,1}^+x_{2,1}^+, x_{1,1}^+x_{2,1}^-, x_{1,1}^-x_{2,1}^- \rangle_{T_*} \subseteq \textnormal{Id}^*(\mathcal{G}_{2,\gamma}^{gri} \oplus W_{\eta_1}^{gri}).$$
Let $f\in \Gamma_n^* \cap \textnormal{Id}^*(\mathcal{G}_{2,\gamma}^{gri} \oplus W_{\eta_1}^{gri})$ be a multihomogeneous identity. Since $\mathcal{G}_{2,\gamma}^{gri} \oplus W_{\eta_1}^{gri}$ has quadratic codimension growth, we have
$$\Gamma_n^* =\Gamma_n^* \cap \textnormal{Id}^*(\mathcal{G}_{2,\gamma}^{gri} \oplus W_{\eta_1}^{gri}) \subseteq I, \quad \textnormal{for all} \quad n\geq 3.$$ 

Moreover, for identities of degree 1, it is straightforward to see that
$$ \Gamma_1^{*}\cap \textnormal{Id}^*(\mathcal{G}_{2,\gamma}^{gri} \oplus W_{\eta_1}^{gri})\subseteq I. $$

It remains to consider the case where $f$ is a multilinear proper identity of degree 2. In this case, after reducing $f$ modulo $I$, we may assume that $f=\alpha x_{1,1}^+x_{2,0}^- +\beta [x_{1,1}^+,x_{2,0}^-]$. 
Considering the evaluation $x_{1,1}^+ \mapsto (e_2,0)$ and $x_{2,0}^- \mapsto (e_1,0)$, we obtain $\alpha +2\beta=0$.
Now, taking the evaluation $x_{1,1}^+ \mapsto (0,e_{13}+e_{24})$ and $x_{2,0}^- \mapsto (0,e_{12}+e_{34})$, we obtain $\alpha =0$.
Therefore, we must have $\alpha=\beta=0$ and thus $f\in I$. These arguments prove that $I=\textnormal{Id}^*(\mathcal{G}_{2,\gamma}^{gri} \oplus W_{\eta_1}^{gri})$.

Now, note that the proper h.w.v.'s $x_{1,1}^+,x_{1,0}^-,x_{1,1}^-$ and $[x_{1,1}^+,x_{2,0}^-]$ are not identities of $\mathcal{G}_{2,\gamma}^{gri} \oplus W_{\eta_1}^{gri}$. By Proposition \ref{prop:non-zero_multiplicities}, the non-zero multiplicities of $\mathcal{G}_{2,\gamma}^{gri} \oplus W_{\eta_1}^{gri}$ satisfy
$$m_{((1)_{1^+})}=m_{((1)_{0^-})}=m_{((1)_{1^-})} = 1 \quad \textnormal{and} \quad  1 \leq m_{((1)_{1^+} ,(1)_{0^-})} \leq 2$$

Moreover, by the previous discussion, there are no non-zero scalars $\alpha$ and $\beta$ such that $\alpha x_{1,1}^+x_{2,0}^- +\beta [x_{1,1}^+,x_{2,0}^-]$ is an identity of $\mathcal{G}_{2,\gamma}^{gri} \oplus W_{\eta_1}^{gri}$. Therefore, by Proposition \ref{prop:non-zero_multiplicities}, we have $m_{((1)_{1^+} ,(1)_{0^-})} = 2$ and the only non-zero proper cocharacters of
$\mathcal{G}_{2,\gamma}^{gri} \oplus W_{\eta_1}^{gri}$ are $\chi_{((1)_{1^-})},\,\chi_{((1)_{0^-})}, \, \chi_{((1)_{1^+})} $ and $2\chi_{((1)_{1^+})} \otimes \chi_{((1)_{0^-})}$.

Using relations~\eqref{codimension}, \eqref{proper_codimension} and~\eqref{cocharacter}, we obtain
\[
c_n^{*}\bigl(\mathcal{G}_{2,\gamma}^{\mathrm{gri}} \oplus W_{\eta_1}^{\mathrm{gri}}\bigr)
= 1 + 3n + 4\binom{n}{2}.
\]\end{proof}

  Before concluding this section, we establish a structural result for unitary $*$-algebras with polynomial codimension growth, which allows us to show that the varieties generated by the unitary algebras above are minimal of quadratic growth.

First, recall that $\widetilde{A}=A\times F$ is the unitary algebra obtained from $A$ through the product $$(a_1,\alpha_1)(a_2,\alpha_2)=(a_1a_2+ \alpha_2a_1+\alpha_1a_2,\alpha_1\alpha_2).$$ Note that, if $A$ is a $*$-algebra, then $\widetilde{A}$ is also a $*$-algebra where the $\mathbb{Z}_2$-grading is given by $\widetilde{A}=(A_0\times F , A_1\times \{0\})$ and the superinvolution is given by $(a,\alpha)^*=(a^*,\alpha)$.

\begin{proposition}\label{unitary_ccg}
    Let $A$ be a unitary $*$-algebra of polynomial codimension growth. If $B\in \textnormal{var}^*(A)$, then either $B$ is unitary or $B\sim_{T_*} N$ or $B\sim_{T_*} C\oplus N$, where $N$ is a nilpotent $*$-algebra and $C$ is a commutative $*$-algebra with trivial superinvolution.
\end{proposition} 
\begin{proof}
     Let $\widetilde{B}$ be the unitary $*$-algebra obtained from $B\in \textnormal{var}^*(A)$. We can decompose $\tilde{B}$ as  $$\widetilde{B}=B_0^+\times F + B_0^-\times \{0\} + B_1^+ \times \{0\} + B_1^- \times \{0\}$$ as a direct sum of subspaces. We start by proving that $\widetilde{B}\in \textnormal{var}^*(A)$. Since ${A}$ and $\widetilde{B}$ are unitary, we consider $f\in \textnormal{Id}^*(A)$ a multilinear proper identity and $\lambda$ the following evaluation on $\tilde{B}$:
     $$\lambda ({f})=f((a_{1,0}^+,\alpha_1),\ldots,(a_{t,0}^+, \alpha_t), (a_{1,0}^-,0),\ldots,(a_{u,0}^-,0), (a_{1,1}^+,0),\ldots,(a_{v,1}^+,0), (a_{1,1}^-,0),\ldots,(a_{w,1}^-,0)),$$ where $a_{i,j}^\epsilon \in B_j^\epsilon$ and $\alpha_i\in F$, for all $j\in \{0,1\}$ and $\epsilon\in\{+,-\}$.

    Since $f$ is a proper polynomial, it is a linear combination of terms of the type $$x_{1,1}^+ \cdots x_{p,1}^+ x_{1,0}^- \cdots x_{q,0}^- x_{1,1}^- \cdots x_{r,1}^- w_1 \cdots w_s,$$ where the $w_i$'s are left normed Lie commutators in the variables from {$X_0^+\cup X_0^-\cup X_1^+ \cup X_1^-$}. Note that the scalars $\alpha_i\in F$ appear only inside commutators. Therefore, we obtain 
    $$\lambda(f)=f((a_{1,0}^+,0),\ldots,(a_{t,0}^+, 0), (a_{1,0}^-,0),\ldots,(a_{u,0}^-,0), (a_{1,1}^+,0),\ldots,(a_{v,1}^+,0), (a_{1,1}^-,0),\ldots,(a_{w,1}^-,0))=0,$$ since $f \in \textnormal{Id}^*({A})$.  Hence, $\widetilde{B}\in \textnormal{var}^*(A)$. 
      
    By ~\cite[Corollary 5.1 and Theorem 5.3]{GIL}, we can consider ${B}$ a finite-dimensional $*$-algebra. If $B$ is already a unitary $*$-algebra, we are done. Suppose then that $B$ is not unitary and assume, as we may, that $F$ is algebraically closed. By~\cite[Theorem 4.1]{GIL2}, we have 
$$B=B_{1}\oplus\cdots\oplus B_{k}+J $$
where $J$ is the Jacobson radical of $B$, $B_{1},\ldots, B_{k}$ are finite-dimensional $*$-simple superalgebras isomorphic to $F$ endowed with trivial superinvolution, and $B_{i}JB_{l}=0$ for all $i\ne l$. Now, if $B_{i}J=JB_{i}=0$ for all $i\ge 1$, then $B=B_{1}\oplus\cdots\oplus B_{k}\oplus J$ and $B$ is the direct sum of a commutative $*$-algebra with trivial superinvolution and a nilpotent algebra or a nilpotent $*$-algebra. In this case, we are done.

Otherwise, there exists $i$ such that $B_{i}J\ne 0$ (or $JB_i \ne 0$). Since $\widetilde{B}=B\times F$, we have
$$\widetilde{B}=\overline{B}_{1}\oplus\cdots\oplus \overline{B}_{k}\oplus\widetilde{F}+\overline{J}$$
where $\overline{B}_{i}=\{(b_{i},0)|b_{i}\in B_{i}\}$, $\widetilde{F}=\{(0,\alpha)|\alpha\in F\}\cong F$ and $\overline{J}=\{(j,0)|j\in J\}$. 
Moreover, since $\overline{B}_{i}\overline{J}\widetilde{F}=\overline{B}_{i}\overline{J}\ne \{0\}$ then, by~\cite[Theorem 2.3]{Ioppololamat}, the sequence $c_n^*(\widetilde{B})$, $n\geq 1$, grows exponentially. But this is a contradiction since $\widetilde{B}\in \textnormal{var}^*(A)$ and ${A}$ has polynomial codimension growth.
\end{proof}

 Let $\mathcal{L}_2$ denote the following collection of $*$-algebras:
\[
\mathcal{L}_2 =
\{\mathcal{G}_{2,\tau},\, \mathcal{G}_{2,\gamma }^{\mathrm{gr}},\, \mathcal{G}_{2,\tau }^{\mathrm{gr}},\, \mathcal{G}_{2,\psi }^{\mathrm{gr}},  \mathcal{G}_{2,\tau }^{\mathrm{gri}},\, \mathcal{G}_{2,\gamma }^{\mathrm{gri}}, C_{3,i_2}, C_{3,i_1}^{\mathrm{gr}},C_{3,i_3}^{\mathrm{gr}}, W_{\eta_2}^{\mathrm{gr}}, W_{\eta_1}^{\mathrm{gri}}, W_{\eta_3}^{\mathrm{gri}}\}.
\]

\begin{proposition}\label{cor_W_minimal}
    The algebras $B\in \mathcal{L}_2$ generate minimal $*$-varieties with quadratic codimension growth. 
\end{proposition}

\begin{proof}
 Assume that $A \in \textnormal{var}^*(W_{\eta_2}^{\mathrm{gr}})$ generates a $*$-variety of quadratic codimension growth. By Proposition~\ref{unitary_ccg}, we may assume without loss of generality that $A$ is a unitary $*$-algebra. We compare the proper cocharacters of $A$ and $W_{\eta_2}^{\mathrm{gr}}$.

Since $A \in \textnormal{var}^*(W_{\eta_2}^{\mathrm{gr}})$, Table~\ref{table_W} implies that the multiplicities of $A$ satisfy $$0
\leq m_{((1)_{1^+})}, m_{((1)_{1^-})} \leq 1\quad \mbox{ and }0 \leq m_{((1)_{1^+},(1)_{1^-})} \leq 1.$$ As $A$ has quadratic codimension growth, it follows that $m_{((1)_{1^+},(1)_{1^-})}=1$, hence $x_{1,1}^+x_{2,1}^- \notin \textnormal{Id}^*(A)$. Consequently, $x_{1,1}^+, x_{1,1}^-\notin \textnormal{Id}^*(A)$, which implies $m_{\lambda}=1$ for all $\lambda \in \{(1)_{1^+},(1)_{1^-}\}$.

Therefore, $c_n^*(A)=c_n^*(W_{\eta_2}^{\mathrm{gr}})$ and, since $A \in \textnormal{var}^*(W_{\eta_2}^{\mathrm{gr}})$, we conclude that $A \sim_{T_*} W_{\eta_2}^{\mathrm{gr}}$. Hence, $W_{\eta_2}^{\mathrm{gr}}$ generates a minimal $*$-variety.

The same argument applies to the others $*$-algebras. \end{proof}

It is important to emphasize that the authors in \cite[Corollary 7.2]{ISDSV} have already proved that the $*$-algebras $\mathcal{G}_{2,\tau}, \mathcal{G}_{2,\tau }^{\mathrm{gri}},\, \mathcal{G}_{2,\gamma }^{\mathrm{gri}}, C_{3,i_2}, W_{\eta_1}^{\mathrm{gri}}$ and $ W_{\eta_3}^{\mathrm{gri}}$ generate minimal varieties of quadratic codimension growth. Therefore, we have provided an alternative proof for these $*$-algebras.

\section{Characterizing varieties with non-zero multiplicities}

This section is devoted to the study of unitary $*$-varieties of quadratic codimension growth. The technique employed consists of establishing a correspondence between the non-zero multiplicities in the proper cocharacters and the algebras generating minimal $*$-varieties in the variety. As a consequence, we obtain a classification of all unitary varieties of quadratic codimension growth and a classification of the minimal ones.

Since we are interested in varieties of quadratic codimension growth, Proposition \ref{Gammancontidoemid} shows that it suffices to analyze the multilinear proper identities of degree $1$ and $2$ of an algebra generating such a variety. More precisely, it is enough to consider the multihomogeneous ones. Accordingly, we decompose $\Gamma_1^*$ and $\Gamma_2^*$ as given in (\ref{Gamma_n}) into their multihomogeneous components as follows

\[
\Gamma_1^*\cong \bigoplus\Gamma_{(1_{t^\epsilon})} \quad \text{and} \quad \Gamma_2^*\cong \Gamma_{(2_{0^+})}\bigoplus (\Gamma_{(2_{t^\epsilon})} \oplus \Gamma_{(1_{0^+},1_{t^\epsilon})}) \bigoplus 2\Gamma_{(1_{s^\delta}, 1_{t^\epsilon})},
\]
where $t^\epsilon, s^\delta \in \{0^-,1^+,1^-\}$ and $t^\epsilon\neq  s^\delta$.

Motivated by Theorem \ref{f+j}, we assume that $A$ is a unitary $*$-algebra of the type $F+J(A)$ generating a variety with quadratic codimension growth. Moreover, we consider 
$$\pi_{n_1,\ldots,n_4}(A)
=
\sum_{(\lambda_1,\ldots,\lambda_4)\vdash (n_1,\ldots,n_4)}
m_{\lambda_1,\ldots,\lambda_4}\,
\chi_{\lambda_1}\otimes \chi_{\lambda_2}\otimes \chi_{\lambda_3}\otimes\chi_{\lambda_4},$$ 
the decomposition of the proper $(n_1,\ldots, n_4)$-cocharacters of $A$, where $(\lambda_1,\ldots,\lambda_4)$ is a multipartition of $(n_1,\ldots,n_4)$ and $n=n_1+\cdots+n_4$ with $n\in\{1,2\}$.

According to Table \ref{tabela_mult}, we observe that the multiplicities above satisfy 
\begin{equation}\label{mult-0or1or2}
    0 \leq m_{ \lambda } \leq 1 \quad \text{and} \quad
    0 \leq m_{((1)_{s^\delta},(1)_{t^\epsilon})} \leq 2,
\end{equation}
where $ \lambda \in \{((1)_{u^{\epsilon_1}}), ((2)_{u^{\epsilon_1}}), (1,1)_{u^{\epsilon_2}}, ((1)_{0^+}, (1)_{u^{\epsilon_1}}) \mid u\in \{1,0\}, \ \epsilon_i \in \{+,-\},\ u^{\epsilon_1} \neq 0^+\}$ and $t^\epsilon,s^\delta\in \{0^-,1^+,1^-\}$ with $s^\delta\neq t^\epsilon $. 

We now proceed to analyze all possible values of the multiplicities described above.

\begin{lemma}\label{lemmaC_2}~ For the multiplicities $m_{((1)_{t^\epsilon})}$, we have:

\begin{enumerate}
  \item $m_{((1)_{0^-})}=1$ if and only if $C_{2,*} \in \textnormal{var}^*(A)$. 
    \item $m_{((1)_{1^+})}=1$ if and only if $C_2^{gr} \in \textnormal{var}^*(A)$. 
    \item $m_{((1)_{1^-})}=1$ if and only if $C_{2,*}^{gr} \in \textnormal{var}^*(A)$. 
\end{enumerate}
\end{lemma}

\begin{proof}
    Suppose that $m_{((1)_{u^\epsilon})}=1$ for some $u^\epsilon \in \{0^-,1^+,1^-\}$. By Proposition \ref{prop:non-zero_multiplicities} and Table \ref{tabela_C2-c3}, there exists a non-zero element $a \in J(A)_u^\epsilon$. Denote by $R$ the $*$-subalgebra of $A$ generated by $1_F$ and $a$, and let $I$ be the $*$-ideal of $R$ generated by $a^2$. Consider the map $\varphi: R/I\rightarrow C_2$ given by $${1_F}\mapsto e_{11}+e_{22} \quad \mbox{ and }\quad {a}\mapsto e_{12}.$$ Note that $\varphi$ defines an isomorphism of $*$-algebras in the following cases:
\begin{enumerate}
    \item[1.] $R/I \cong C_{2,*}$ if $u^\epsilon=0^-$;
    \item[2.]    $R/I \cong C_{2}^{gr}$ if $u^\epsilon=1^+$;
    \item[3.] $R/I \cong C_{2,*}^{gr}$ if $u^\epsilon=1^-$.
\end{enumerate}

Moreover, if $C_{2,*} \in \textnormal{var}^*(A)$ then, by Remark \ref{remark} and Table \ref{tabela_C2-c3}, $m_{((1)_{0^-})}=1$. Similarly, for  $C_{2}^{gr} \in \textnormal{var}^*(A)$ we have $m_{((1)_{1^+})}=1$ and if $C_{2,*}^{gr} \in \textnormal{var}^*(A)$ then $m_{((1)_{1^-})}=1$.
\end{proof}

\begin{lemma}\label{lemmaC_3}  ~ For the multiplicities $m_{((2)_{t^\epsilon})}$, we have:

\begin{enumerate}
 \item $m_{((2)_{0^-})} =1$ if and only if $C_{3,i_2} \in \textnormal{var}^*(A)$.
    \item $m_{((2)_{1^+})}=1$ if and only if $C_{3,i_1}^{gr} \in \textnormal{var}^*(A)$.
    \item $m_{((2)_{1^-})} =1$ if and only if $C_{3,i_3}^{gr} \in \textnormal{var}^*(A)$.  
\end{enumerate}
\end{lemma}

\begin{proof}

Assume \( m_{((2)_{u^\epsilon})} =1 \), for some \( u^\epsilon \in \{0^-,1^+,1^-\} \). By Proposition \ref{prop:non-zero_multiplicities} and Table \ref{tabela_mult}, it follows that \( (x_{1,u}^\epsilon)^2 \notin \textnormal{Id}^*(A) \). Hence, there exists a non-zero element \( a \in J(A)_u^\epsilon \) such that \( a^2 \neq 0 \). Let \( R \) be the \(*\)-subalgebra of \( A \) generated by \( 1_F \) and \( a \), and let \( I \) denote the $*$-ideal of $R$ generated by \( a^3 \). The map $\varphi: R/I\rightarrow C_3$ given by 
$$1\mapsto e_{11}+e_{22}+e_{33} ,\quad a\mapsto e_{12}+e_{23} \quad \mbox{ and }\quad a^2\mapsto e_{13}$$
defines an isomorphism of $*$-algebras in the following  cases:
\begin{enumerate}
    \item[1.] if  \( u^\epsilon=0^- \)  then  \( R/I \cong C_{3,i_2} \).
    \item[2.]  if \( u^\epsilon=1^+ \) then \( R/I \cong C_{3,i_1}^{gr} \).
    \item[3.] if \( u^\epsilon=1^- \) then \( R/I \cong C_{3,i_3}^{gr} \). 
\end{enumerate}
The converse in each case follows from  Table \ref{tabela_C2-c3} and Remark \ref{remark}.
\end{proof}

\begin{lemma}\label{lemmaU_3,G_2}~ For the multiplicities $m_{((1,1)_{t^\epsilon})}$, we have:

    \begin{enumerate}
        \item $m_{((1,1)_{0^+})}=1 $ if and only if $U_{3,*} \in \textnormal{var}^*(A)$.
        \item $m_{((1,1)_{0^-})} =1 $ if and only if $\mathcal{G}_{2,\tau }\in \textnormal{var}^*(A)$.
        \item $m_{((1,1)_{1^+})}=1 $ if and only if $\mathcal{G}_{2,\psi}^{gr} \in \textnormal{var}^*(A)$.
        
        \item $m_{((1,1)_{1^-})} =1 $ if and only if $\mathcal{G}_{2,\tau }^{gr}\in \textnormal{var}^*(A)$.
    \end{enumerate}
\end{lemma}
\begin{proof}
    Assume that $m_{((1,1)_{0^\epsilon})}=1$, for some $\epsilon \in \{+,-\}$. By Proposition \ref{prop:non-zero_multiplicities} we have $[x_{1,0}^+,x_{2,0}^+]\not\equiv 0$ or $[x_{1,0}^-,x_{2,0}^-]\not\equiv 0$ on $A$ and then,  by \cite[Lemmas 4.3 and 4.4]{Wesley}, we have that
    
    \begin{enumerate}
        \item[1.] if $\epsilon=+$ then $U_{3,*}\in \textnormal{var}^*(A)$;
        \item[2.] if $\epsilon=-$ then $\mathcal{G}_{2,\tau} \in \textnormal{var}^*(A)$.
    \end{enumerate}
    
    Suppose now that $m_{((1,1)_{1^+})}=1$. By Proposition \ref{prop:non-zero_multiplicities} we get $[x_{1,1}^+,x_{2,1}^+] \not\equiv  0$ on $A$. 
    Since $A$ and $\mathcal{G}_{2,\psi}^{gr}$ are unitary $*$-algebras with quadratic codimension growth, we have $$\Gamma_n^*=\Gamma_n^* \cap\textnormal{Id}^*(A)=\Gamma_n^*\cap \textnormal{Id}^*(\mathcal{G}_{2,\psi}^{gr}), \quad \textnormal{for all} \quad n\geq 3.$$ 
    Since $x_{1,0}^+,x_{1,1}^+ \notin \textnormal{Id}^*(\mathcal{G}_{2,\psi}^{gr})$ and $[x_{1,1}^+,x_{2,1}^+] \neq 0$ on $A$, we have $x_{1,0}^+,x_{1,1}^+\notin \textnormal{Id}^*(A) \cap \textnormal{Id}^*(\mathcal{G}_{2,\psi}^{gr})$ and so $$\Gamma_1^*\cap \textnormal{Id}^*(A)\subseteq \Gamma_1^* \cap \textnormal{Id}^*(\mathcal{G}_{2,\psi}^{gr}).$$
    
   Finally, let $f \in \Gamma_2^* \cap\textnormal{Id}^*(A)$ be a multihomogeneous identity of $A$. By Lemma \ref{codim_G_2}, the polynomials 
\[
    x_{1,0}^-, \quad x_{1,1}^-, \quad [x_{1,0}^+,x_{2,0}^+], \quad [x_{1,0}^+,x_{1,1}^+] \quad x_{1,1}^+\circ x_{2,1}^+, \quad x_{1,1}^+x_{2,1}^+x_{3,1}^+
\]
are identities of $\mathcal{G}_{2,\psi}^{gr}$. Consequently, we have either
\[ f\in \textnormal{Id}^*(\mathcal{G}_{2,\psi}^{gr})\quad \mbox{ or }\quad 
    f =\alpha x_{1,1}^+x_{2,1}^+ + \beta[x_{1,1}^+,x_{2,1}^+],
\] where $\alpha$ and $\beta$ are non-zero. In the second case, we obtain
\[
    \gamma x_{1,1}^+x_{2,1}^+ - x_{2,1}^+x_{1,1}^+\equiv 0 \mbox{ on }A, \quad \text{where } \gamma = \frac{\alpha+\beta}{\beta}.
\]
Applying the superinvolution on the polynomial above, we obtain $-\gamma x_{2,1}^+x_{1,1}^+ +x_{1,1}^+x_{2,1}^+\equiv 0$ on $A$. Combining both relations, we obtain $(1-\gamma^2) x_{1,1}^+x_{2,1}^+ \equiv 0$ on $A$. Since $x_{1,1}^+x_{2,1}^+ \not\equiv 0$ on $A$, then $\gamma^2=1$. We analyze the two possible cases for $\gamma$:
\begin{enumerate}
    \item[1.] If $\gamma=-1$, then we have  $f=-
    \beta x_{1,1}^+ \circ x_{2,1}^+ $ and so $f\in \textnormal{Id}^*(\mathcal{G}_{2,\psi}^{gr})$.
    \item[2.] If $\gamma=1$, then $\alpha =0$, a contradiction.
\end{enumerate}

Hence, we conclude that $f \in \textnormal{Id}^*(\mathcal{G}_{2,\psi}^{gr})$. Therefore, $\Gamma_2^{*}\cap\textnormal{Id}^*(A) \subseteq \Gamma_2^{*} \cap \textnormal{Id}^*(\mathcal{G}_{2,\psi}^{gr})$ and so $\mathcal{G}_{2,\psi}^{gr}\in \textnormal{var}^*(A)$.
    
   Similarly, if $m_{((1,1)_{1^-})}=1$, we can use Lemma \ref{codim_G_2} and the previous argument to prove that $\mathcal{G}_{2,\tau}^{gr}\in \textnormal{var}^*(A)$.

   The converse follows by Remark \ref{remark} and Table \ref{tableG_2}.
\end{proof}

\begin{lemma}\label{lemmaN_3,U_3}~ For the multiplicities $m_{((1)_{0^+},(1)_{t^\epsilon})}$, we have:

\begin{enumerate}
    \item  $m_{((1)_{0^+},(1)_{1^+})} =1$ if and only if $U_{3,*}^{gri} \in \textnormal{var}^*(A)$.
    \item  $m_{((1)_{0^+},(1)_{0^-})}=1$ if and only if  $N_{3,*} \in \textnormal{var}^*(A)$.
    \item  $m_{((1)_{0^+},(1)_{1^-})}=1$ if and only if $N_{3,*}^{gri} \in \textnormal{var}^*(A)$.
\end{enumerate}
\end{lemma}
\begin{proof}
    Observe that if $m_{((1)_{0^+}, (1)_{0^-})}=1$ then, by Proposition \ref{prop:non-zero_multiplicities}, we have $[x_{1,0}^+,x_{2,0}^-]\notin \textnormal{Id}^* (A)$. In this case, by \cite[Lemmas 4.3]{Wesley}, we obtain  $N_{3,*}\in \textnormal{var}^*(A)$.
    
    Now, we assume that $m_{((1)_{0^+}, (1)_{1^+})}=1$. By Proposition \ref{prop:non-zero_multiplicities} we have $[x_{1,0}^+,x_{2,1}^+]\notin \textnormal{Id}^* (A)$. Since $U_{3,*}^{gri}$ and $A$ are unitary $*$-algebras with quadratic codimension growth, we have
    $$\Gamma_{n}= \Gamma_{n}\cap \textnormal{Id}^*(A) = \Gamma_n^{*} \cap \textnormal{Id}^*(U_{3,*}^{gri}), \quad \textnormal{for all} \quad n\geq 3. $$
    Furthermore,  since $[x_{1,0}^+,x_{2,1}^+]\notin \textnormal{Id}^* (A)$ then $x_{1,1}^+,x_{1,1}^- \notin \textnormal{Id}^*(A)$. Since $x_{1,0}^-
    \in \textnormal{Id}^*(U_{3,*}^{gri})$ then
    $$ \Gamma_1^{*} \cap \textnormal{Id}^*(A) \subseteq \Gamma_1^{*}\cap \textnormal{Id}^*(U_{3,*}^{gri}). $$
   
    Therefore, we just need to analyze the proper multilinear identities of degree $2$ of $A$. Let $f \in \Gamma_2^* \cap\textnormal{Id}^*(A)$ be a multihomogeneous identity. By Lemma \ref{T-ideal-U_and_N},  the polynomials 
    $$x_{1,0}^-, \quad x_{1,1}^+x_{2,1}^+, \quad x_{1,1}^-x_{2,1}^-, \quad x_{1,1}^+x_{2,1}^-, \quad [x_{1,1}^-,x_{2,0}^+]$$
    are identities of $U_{3,*}^{gri}$. Then, either $f\in \textnormal{Id}^*(U_{3,*}^{gri})$ or $f \in \Gamma_{(1_{0^+}, 1_{1^+})}$. Since $\Gamma_{(1_{0^+}, 1_{1^+})}$ is linearly generated by $[x_{1,0}^+,x_{2,1}^+]$ and it is a non-identity of $A$, then we have $f\in \textnormal{Id}^*(U_{3,*}^{gri})$. Therefore, $ \Gamma_2^{*}\cap\textnormal{Id}^*(A) \subseteq \Gamma_2^{*} \cap \textnormal{Id}^*(U_{3,*}^{gri}) $ and so $U_{3,*}^{gri}\in \textnormal{var}^*(A)$. 
    
   Similarly, if $m_{((1)_{0^+},(1)_{1^-})} =1$, we can use Lemma \ref{T-ideal-U_and_N} and the previous argument to prove that $N_{3,*}^{gri}\in \textnormal{var}^*(A)$.

   The converse follows by Remark \ref{remark} and Table \ref{tabela_U3-N3}.
\end{proof}

\begin{lemma}\label{lemmaG_2,W}~ For the multiplicities $m_{(1)_{1^+},((1)_{0^-})},m_{((1)_{1^+},(1)_{1^-})}$ and $m_{((1)_{0^-},(1)_{1^-})}$, we have:

\begin{enumerate}

    \item  $m_{((1)_{0^-},(1)_{1^+})}=1$ then  $\mathcal{G}_{2,\gamma}^{gri} \in \textnormal{var}^*(A)$ or $W_{\eta_1}^{gri} \in \textnormal{var}^*(A)$.
    \item  $m_{((1)_{1^+},(1)_{1^-})}=1$ then $\mathcal{G}_{2,\gamma}^{gr} \in \textnormal{var}^*(A)$ or $W_{\eta_2}^{gr~} \in \textnormal{var}^*(A)$.
    \item  $m_{((1)_{0^-},(1)_{1^-})}= 1$ then $\mathcal{G}_{2,\tau}^{gri} \in \textnormal{var}^*(A)$ or $W_{\eta_3}^{gri} \in \textnormal{var}^*(A)$.
\end{enumerate}
\end{lemma} 
\begin{proof} 
    Assume $m_{((1)_{0^-},(1)_{1^+})}= 1$. Since $x_{1,0}^-x_{2,1}^+$ and $[x_{1,0}^-,x_{2,1}^+]$ are the proper h.w.v.'s associated to the multipartition $((1)_{0^-},(1)_{1^+})$ of $2$, then, by Proposition \ref{prop:non-zero_multiplicities}, there exists $(\alpha,\beta)\neq (0,0)$ such that $$\alpha x_{1,0}^-x_{2,1}^++ \beta[x_{1,0}^-,x_{2,1}^+]\equiv 0.$$ Moreover, since $*$ is a superinvolution, if $x_{1,0}^-x_{2,1}^+\equiv 0$ then $x_{2,1}^+x_{1,0}^-\equiv 0$, therefore, only the following cases can occur:
   \begin{enumerate}
        \item[1.]  $\alpha= 0$ and $\beta\neq 0$;
        \item[2.]  $\alpha \neq 0$ and $\beta\neq0$.
    \end{enumerate} 
    
    In the first case, note that since $x_{1,0}^-x_{2,1}^+ \not\equiv 0$ and $[x_{1,0}^-,x_{2,1}^+]\equiv0$ then there exist $a\in J(A)_{0}^-$ and $b\in J(A)_{1}^+$ such that $ab \neq 0$ and $ab=ba$. Let $R$ be the $*$-subalgebra of $A$ generated by $1_F, a, b$ and $I$ the $*$-ideal of $R$ generated by $a^2$ and $b^2$. Since $A$ has quadratic codimension growth, we have $\Gamma_{n}^{*} \subseteq \textnormal{Id}^*(A)$, for all $n\geq3$ and then $I$ is linearly generated by $a^2$ and $b^2$. 
    Moreover, since $I \subseteq A_0$ then $ab \notin I$ and so $a,b, ab \notin I$. Now, we observe that the linear map $\varphi: R/I \rightarrow W$ given by
    $$\overline{1_F} \mapsto e_{11}+\cdots +e_{44}, \quad \overline{a} \mapsto e_{12}+e_{34}, \quad \overline{b}\mapsto e_{13}+e_{24}, \quad \overline{ab} \mapsto e_{14}$$ defines an $*$-isomorphism between $R/I$  and $W_{\eta_1}^{gri}$.

    In the second case, we have $\gamma x_{1,0}^-x_{2,1}^+\equiv x_{2,1}^+x_{1,0}^-$ on $A$, where $\gamma=\frac{\alpha+\beta}{\beta}$. Since $*$ is a superinvolution, it follows that
  $$(\gamma^2-1) x_{1,0}^-x_{2,1}^+\equiv 0\mbox{ on }A,$$ thus $\gamma^2=1$. Since $\alpha \neq 0$ then $\gamma=-1$. Therefore $  x_{1,0}^-x_{2,1}^+ +x_{2,1}^+x_{1,0}^- \equiv 0$ on $A$. 
  
  Consider $a\in J(A)_0^-$ and $b\in J(A)_{1}^+$ such that $ab\neq 0$. Since $  x_{1,0}^-x_{2,1}^+ +x_{2,1}^+x_{1,0}^- \equiv 0$ on $A$, then $ab= -ba$. Let $R$ be the $*$-subalgebra of $A$ generated by $1_F,a,b$ and $I$ the $*$-ideal of $R$ generated by $a^2$ and $b^2$. Hence, $R/I$ is linearly generated by the non-zero elements $\overline{1_F}$, $\overline{a}$, $\overline{b}$ and $\overline{ab}$ satisfying $\overline{a}^2=\overline{b}^2=\overline{ab}+\overline{ba}=0$. Consider the linear map $\phi: R/I \rightarrow \mathcal{G}_2$ given by
    $$1_F \mapsto 1, \quad a\mapsto e_1, \quad b\mapsto e_2, \quad ab \mapsto e_1e_2.$$ It is straightforward to see that $\phi$ is an isomorphism of $*$-algebras between $R/I$ and $\mathcal{G}_{2,\gamma}^{gri}$. 

    In a similar way, we can prove the items 2) and 3).
\end{proof}

Recall that, according to~\eqref{mult-0or1or2}, the multiplicities considered in the previous lemma may also assume the value~$2$. In this case, we obtain the following result.

\begin{lemma}\label{lemma_mult2} ~ For the multiplicities $m_{((1)_{0^-}, (1)_{1^+})},m_{((1)_{1^+},(1)_{1^-})}$ and $m_{((1)_{0^-},(1)_{1^-})}$, we have:

\begin{enumerate}
    \item  $m_{((1)_{0^-},(1)_{1^+})}= 2$ if and only if $\mathcal{G}_{2,\gamma}^{gri} \oplus W_{\eta_1}^{gri} \in \textnormal{var}^*(A)$. 
    
    \item  $m_{((1)_{1^+},(1)_{1^-})}=2$ if and only if $\mathcal{G}_{2,\gamma}^{gr} \oplus W_{\eta_2}^{gr} \in \textnormal{var}^*(A)$.
    
    \item  $m_{((1)_{0^-},(1)_{1^-})}=2$ if and only if $\mathcal{G}_{2,\tau}^{gri} \oplus W_{\eta_3}^{gri} \in \textnormal{var}^*(A)$.
\end{enumerate}
\end{lemma} 
\begin{proof}
   
Assume that $m_{((1)_{0^-},(1)_{1^+})}=2$. Since both $A$ and
$\mathcal{G}_{2,\gamma}^{gri} \oplus W_{\eta_1}^{gri}$ are unitary $*$-algebras
with quadratic codimension growth, it follows that
\[ \Gamma_n^*=
\Gamma_n^* \cap \textnormal{Id}^*(A)
=
\Gamma_n^* \cap \textnormal{Id}^*(\mathcal{G}_{2,\gamma}^{gri} \oplus W_{\eta_1}^{gri}),
\quad \text{for all } n \geq 3.
\]

Moreover, by Proposition~\ref{prop:non-zero_multiplicities}, there are no linear combinations of the polynomials
$[x_{1,0}^-,x_{2,1}^+]$ and $x_{1,0}^-x_{2,1}^+$ resulting in an identity of $A$. In particular, $x_{1,0}^-\circ x_{2,1}^+ \notin \textnormal{Id}^*(A)$. Consequently, the skew homogeneous component of degree $1$ of $A$ is non-trivial and so we obtain  $x_{1,0}^-,x_{1,1}^+, x_{1,1}^-\notin \textnormal{Id}^*(A)$. Therefore, $A$ has no identity of degree $1$. 

Let $f \in \Gamma_2^* \cap \textnormal{Id}^*(A)$ be a multihomogeneous identity. Observing the generators of the $T_*$-ideal of $\mathcal{G}_{2,\gamma}^{gri} \oplus W_{\eta_1}^{gri}$ given in
Lemma~\ref{codim_G2_sum_W}, we have either
\[
f \in \textnormal{Id}^*(\mathcal{G}_{2,\gamma}^{gri} \oplus W_{\eta_1}^{gri})
\quad \text{or} \quad
f = \alpha x_{1,0}^-x_{1,1}^+ + \beta [x_{1,0}^-,x_{1,1}^+],
\]
for some $\alpha,\beta \in F$. In the latter case, as observed above, there are no linear combinations
of the polynomials $x_{1,0}^-x_{1,1}^+$ and $[x_{1,0}^-,x_{1,1}^+]$ resulting in an identity of $A$, which leads to a contradiction. Therefore, this case cannot occur. Consequently, we conclude that
\[
\Gamma_n^* \cap \textnormal{Id}^*(A)
\subseteq
\Gamma_n^* \cap \textnormal{Id}^*(\mathcal{G}_{2,\gamma}^{gri} \oplus W_{\eta_1}^{gri}),
\] for all $n$, which completes the proof.

Items~(2) and~(3) follow by analogous arguments. 

The converse statements in each
case are obtained from Remark~\ref{remark} and Table~\ref{table_sum}.

\end{proof}

Before we present the main result of this section, we define $\mathcal{I}$ as the following set of $*$-algebras:

\[
\begin{aligned}
\mathcal{I} = \{&F,\
C_2^{gr},\ C_{2,*},\ C_{2,*}^{gr},
C_{3,i_1}^{gr},\ C_{3,i_2},\ C_{3,i_3}^{gr},  U_{3,*},\ U_{3,*}^{gri},\ N_{3,*},\ N_{3,*}^{gri}, \\
& \mathcal{G}_{2,\psi}^{gr},\ \mathcal{G}_{2,\tau},\ \mathcal{G}_{2,\tau}^{gr},\ \mathcal{G}_{2,\tau}^{gri},  \mathcal{G}_{2,\gamma}^{gr},\ \mathcal{G}_{2,\gamma}^{gri}, 
 W_{\eta_1}^{gri},\ W_{\eta_2}^{gr},\ W_{\eta_3}^{gri}
\,\}.
\end{aligned}
\]

Also, we consider
$$\widetilde{\mathcal{I}}=\mathcal{I}-\{ F,\ C_2^{gr},\
C_{2,*},\
C_{2,*}^{gr}\}.$$

\begin{theorem}\label{Theorem_final}
    Let $A$ be a unitary $*$-algebra over a field $F$ of characteristic zero with quadratic codimension growth. Then $A$ is $T_*$-equivalent to a finite direct sum of $*$-algebras in the set $\mathcal{I}$ with at least one direct summand belonging to $\widetilde{\mathcal{I}}$. 
\end{theorem}

\begin{proof}

Since $A$ has polynomial growth, by~\cite[Corollary 5.1 and Theorem 5.3]{GIL}, we may assume that $A$ is a finite-dimensional $*$-algebra. Moreover, by Theorem~\ref{f+j}, the algebra $A$ is $T_*$-equivalent to a finite direct sum
\[
A_1 \oplus \cdots \oplus A_m,
\]
where each $A_i$ is either nilpotent or of the form $F+J(A_i)$. Since $A$ has quadratic codimension growth, there exists an index $i\in\{1,\ldots,m\}$ such that $A_i$ has quadratic codimension growth and is of the form $F+J(A_i)$. Fix such an index $i$. By Proposition~\ref{unitary_ccg}, we may further assume that $A_i$ is a unitary $*$-algebra. In this case, we have
\[
\Gamma_n^* \subseteq \textnormal{Id}^*(A_i), \qquad \text{for all } n\geq 3.
\]

We now analyze the multiplicities of the proper $(n_1,\ldots,n_4)$-cocharacters $\pi_{n_1,\ldots,n_4}(A_i)$ for $n=n_1+\cdots+n_4$ with $n\in \{1,2\}$, which satisfy relation~\eqref{mult-0or1or2}.  

For $n=1$, by Lemma~\ref{lemmaC_2}, the condition $m_{((1)_{0^-})}\neq 0$ implies that $C_{2,{*}}\in\textnormal{var}^*(A_i)$. Similarly, if $m_{((1)_{1^+})}\neq 0$, then $C_2^{gr}\in\textnormal{var}^*(A_i)$, while the condition $m_{((1)_{1^-})}\neq 0$ yields $C_{2,*}^{gr}\in\textnormal{var}^*(A_i)$. 

We now consider the case $n=2$ and analyze the multiplicities associated with
the partitions of $2$. For partitions of the form $((2)_{s^\epsilon})$, the only
multiplicities that may be non-zero are
$m_{((2)_{0^-})}$, $m_{((2)_{1^+})}$, and $m_{((2)_{1^-})}$. In these cases,
Lemma~\ref{lemmaC_3} ensures, respectively, that
$C_{3,i_2}$, $C_{3,i_1}^{gr}$, or $C_{3,i_3}^{gr}$ belongs to
$\var^*(A_i)$.

Next, consider multiplicities corresponding to partitions of type
$((1,1)_{s^\epsilon})$. By Lemma~\ref{lemmaU_3,G_2}, if
$m_{((1,1)_{0^+})}\neq 0$, then $U_{3,*}\in\var^*(A_i)$; if
$m_{((1,1)_{0^-})}\neq 0$, then $\mathcal{G}_{2,\tau}\in\var^*(A_i)$; if
$m_{((1,1)_{1^+})}\neq 0$, then $\mathcal{G}_{2,\psi}^{gr}\in\var^*(A_i)$; and
finally, if $m_{((1,1)_{1^-})}\neq 0$, then
$\mathcal{G}_{2,\tau}^{gr}\in\var^*(A_i)$.

We now turn to partitions of the form $((1)_{0^+},(1)_{s^\epsilon})$. In this
situation, the only possible non-zero multiplicities are
$m_{((1)_{0^+},(1)_{1^+})}$, $m_{((1)_{0^+},(1)_{0^-})}$, and
$m_{((1)_{0^+},(1)_{1^-})}$. According to Lemma~\ref{lemmaN_3,U_3}, these cases
imply, respectively, that $U_{3,*}^{gri}$, $N_{3,*}$ or
$N_{3,*}^{gri}$ belongs to $\var^*(A_i)$.

The remaining possibilities correspond to the multiplicities
$m_{((1)_{0^-},(1)_{1^+})}$, $m_{((1)_{1^+},(1)_{1^-})}$, and
$m_{((1)_{0^-},(1)_{1^-})}$, each of which may take the values $0$, $1$, or $2$.
If any of these multiplicities is equal to $1$, Lemma~\ref{lemmaG_2,W} implies
that, respectively,
\[
\mathcal{G}_{2,\gamma}^{gri} \ \text{or}\ W_{\eta_1}^{gri}, \qquad
\mathcal{G}_{2,\gamma}^{gr} \ \text{or}\ W_{\eta_2}^{gr}, \qquad
\mathcal{G}_{2,\tau}^{gri} \ \text{or}\ W_{\eta_3}^{gri}
\]
belongs to $\var^*(A_i)$.

If one of the above multiplicities is equal to $2$, then
Lemma~\ref{lemma_mult2} shows that, in the corresponding cases,
\[
\mathcal{G}_{2,\gamma}^{gri}\oplus W_{\eta_1}^{gri}, \qquad
\mathcal{G}_{2,\gamma}^{gr}\oplus W_{\eta_2}^{gr}, \qquad
\mathcal{G}_{2,\tau}^{gri}\oplus W_{\eta_3}^{gri}
\]
belongs to $\var^*(A_i)$.

For each non-zero value of the multiplicity $m_\lambda$, where $\lambda=(\lambda_1,\ldots,\lambda_4)\vdash(n_1,\ldots,n_4)$ with $n=n_1+\cdots+n_4$ and $n\in \{1,2\}$, we denote by $A_{i,m_\lambda}$ the corresponding $*$-algebra belonging to $\textnormal{var}^*(A_i)$, as listed above. Let
\[
B=\bigoplus_{m_\lambda\neq 0} A_{i,m_\lambda}
\]
be the direct sum of all such algebras. Clearly, by the discussion above, we have $B\in\textnormal{var}^*(A_i)$. Moreover, by Remark~\ref{remark} and Tables~\ref{tabela_U3-N3}~-~\ref{table_sum},
the algebras $A_i$ and $B$ have the same multiplicities $m_\lambda$ in the decomposition of all proper $(n_1,\ldots,n_4)$-cocharacters. Hence, $c_n^*(A_i)=c_n^*(B)$, for all $n$, and consequently $A_i\sim_{T_*}B$.

Since $A_i$ has quadratic codimension growth, at least one multiplicity associated to a partition of $2$ must be non-zero. This proves that at least one $*$-algebra in $\widetilde{\mathcal{I}}$ appears as a direct summand of $B$ and therefore of $A$.

We apply the preceding argument to every component $A_i$ having quadratic codimension growth.

Finally, if some $*$-algebra $A_k$ has at most linear codimension growth, for some $1\leq k\leq m$, then by~\cite[Theorem~7.2]{Ioppololamat}, $A_k$ is $T_*$-equivalent to a finite direct sum of the $*$-algebras
\[
N,\; F,\; C_2^{gr},\; C_{2,*},\; C_{2,*}^{gr}.
\]

Finally, recall that $A \sim_{T_*} A_1 \oplus \cdots \oplus A_m$. After reordering the $*$-algebras $A_i$, we may assume that
\[
A \sim_{T_*} A_1 \oplus \cdots \oplus A_k \oplus N,
\]
where each $A_i$ is a finite direct sum of algebras belonging to
$\mathcal{I}$ and $N$ is a nilpotent $*$-algebra. Since
$A_1 \oplus \cdots \oplus A_k \oplus N$ has quadratic codimension growth,
Proposition~\ref{unitary_ccg} implies that it is unitary. Consequently,
$N=\{0\}$, which completes the proof.

\black 

\end{proof}

Finally, as a consequence of Theorem \ref{Theorem_final} we provide the classification of the minimal unitary $*$-varieties with quadratic codimension growth.

\begin{corollary}
     Let $A$ be a unitary $*$-algebra with quadratic codimension growth. Then $A$ generates a minimal variety if and only if $A$ is $T_*$-equivalent to some algebra $B\in \widetilde{\mathcal{I}}$.
\end{corollary}

In \cite{Ioppololamat}, the authors provided a complete classification of the varieties generated by $*$-algebras with at most linear codimension growth. In fact, they proved that any such variety is generated by a finite direct sum of $*$-algebras generating minimal varieties, each one having at most linear growth. In the previous theorem, we provided a complete classification of unitary $*$-algebras generating varieties of quadratic codimension growth. As a consequence of the previous results, we have the following. 

\begin{corollary}
    Let $A$ be a unitary $*$-algebra. Then $c_n ^*(A)\leq \alpha n^2$ if and only if $A$ is $T_*$-equivalent to a finite direct sum of algebras generating minimal varieties with at most quadratic codimension growth. 
\end{corollary}


\end{document}